\RequirePackage{fix-cm}
\RequirePackage{amsmath}
\documentclass[smallextended, nospthms]{svjour3}       
\usepackage[utf8]{inputenc}
\usepackage[english]{babel}
\usepackage{amsthm}
\usepackage{amsfonts}
\usepackage{amssymb}
\usepackage{graphicx}
\usepackage{cite}
\usepackage[hidelinks, colorlinks=true, citecolor=blue, linkcolor=blue]{hyperref}
\usepackage{xcolor}
\usepackage{bm}
\usepackage[font=footnotesize]{caption}
\usepackage{subcaption}
\usepackage{stmaryrd}
\usepackage{subfiles}
\usepackage{hyperref}
\usepackage{mathtools}
\usepackage{mathrsfs} 
\usepackage[export]{adjustbox} 
\usepackage{mathabx}

\usepackage{cleveref} 

\usepackage{pgfplots}
\usepackage{tikz}

\usepackage{tabstackengine}
\stackMath

\numberwithin{equation}{section}

\newtheoremstyle{cited}
{3pt}
{3pt}
{\itshape}
{}
{\bfseries}
{.}
{.5em}
{\thmname{#1} \thmnumber{#2} \thmnote{\normalfont#3}}

\theoremstyle{cited}
\newtheorem{citedthm}{Theorem}[section]

\newtheorem{remark}{Remark}[section]
\newtheorem{theorem}{Theorem}[section]

\newtheorem{lemma}[theorem]{Lemma}

\newcommand{\grad}{{\operatorname{grad}}}
\renewcommand{\div}{{\operatorname{div}}}
\newcommand{\jump}[1]{\llbracket #1 \rrbracket }

\newcommand{\vertiii}[1]{{\left\vert\kern-0.25ex\left\vert\kern-0.25ex\left\vert #1 \right\vert\kern-0.25ex\right\vert\kern-0.25ex\right\vert}}

\smartqed  
\begin{document}
	
	
	\title{
		Block Preconditioners for Mixed-dimensional Discretization of Flow in Fractured Porous Media
		\thanks{
			The first author acknowledges the financial support from the TheMSES project funded by Norwegian Research Council grant 250223. The work of the second author is partially supported by the National Science Foundation under grant DMS-1620063.
		}
	}
	\titlerunning{Preconditioners for Mixed-dimensional Flow in Fractured Porous Media}
	
	\author{Ana Budi\v{s}a \and Xiaozhe Hu}
	
	\date{Submission date: May 31, 2019}
	
	\institute{
		A. Budi\v{s}a \at
		Department of Mathematics, University of Bergen, P. O. Box 7800, N-5020 Bergen, Norway. \\
		\email{Ana.Budisa@uib.no}
		\and
		X. Hu \at
		Department of Mathematics, Tufts University, 503 Boston Ave, Medford, MA 02155, USA. \\
		\email{Xiaozhe.Hu@tufts.edu}%
	}
		
	\maketitle
	
	\begin{abstract}
	In this paper, we are interested in an efficient numerical method for the mixed-dimensional approach to modeling single-phase flow in fractured porous media. The model introduces fractures and their intersections as lower-dimensional structures, and the mortar variable is used for flow coupling between the matrix and fractures. We consider a stable mixed finite element discretization of the problem, which results in a parameter-dependent linear system. For this, we develop block preconditioners based on the well-posedness of the discretization choice. The preconditioned iterative method demonstrates robustness with regards to discretization and physical parameters. 
	The analytical results are verified on several examples of fracture network configurations, and notable results in reduction of number of iterations and computational time are obtained.
	\end{abstract}

	
	\section{Introduction}
	\label{sec:intro}
	
	Fracture flow has become a case of intense study recently due to many possible subsurface applications, such as CO$ _2 $ sequestration or geothermal energy storage and production. It has become clear that the dominating role of fractures in the flow process in the porous medium calls for reexamination of existing mathematical models, numerical methods and implementations in these cases. 
	
	Considering modeling and analysis, a popular and effective development is reduced fracture models \cite{karimifard, frih:roberts, boon:nordbotten:yotov} that represent fractures and fracture intersections as lower-dimensional manifolds embedded in a porous medium domain. The immediate advantages of such modeling are in more accurate representation of flow patterns, especially in case of highly conductive fractures, and easier handling of discontinuities over the interfaces. This has also allowed for implementation of various discretization methods, from finite volume methods \cite{karimifard, sandve2012} to (mixed) finite element methods \cite{frih:roberts} and other methods \cite{fumagalli2019, formaggia:scotti2018}. These methods mostly differ in two aspects: whether the fractures conform to the discrete grid of the porous medium \cite{boon:nordbotten:yotov} or are placed arbitrarily within the grid \cite{formaggia2014, scotti_xfem, Schwenck2015}, or whether pressure or flux continuity is preserved. Comparison studies of different discretization methods and their properties can be found in \cite{berre2018, Flemisch2016a, Nordbotten2018}.
	
	Although there is a wide spectrum of discretization methods, little has been done to develop robust and efficient solvers. This aspect of implementation can be very important since applications of fractured porous media usually include large-scale simulations of subsurface reservoirs and the resulting discretized linear systems of equations can become ill-conditioned and quite difficult to solve. The linear system represents a discrete version of the partial differential equation (PDE) operator that has unbounded spectrum.  Thus, its condition number tends to infinity when the mesh size is approaching zero. Moreover, the variability of the physical parameters, such as the permeabilities and aperture, can additionally influence the scale of the condition number of the system. Instead of using direct methods, we consider Krylov subspace iterative methods to solve such large scale problems.  Since the convergence rate of the Krylov subspace methods depends on the condition number of the system, suitable preconditioning techniques are usually required to achieve a good performance. A recent study on a geometric multigrid method \cite{arraras:paco} for the fracture problem shows how standard iterative methods can be extended and perform well on mixed-dimensional discretizations, but still there are limitations that need to be overcome for general fractured porous media simulations. 
	
	In this paper, we aim to provide a general approach to preconditioning the mixed-dimensional flow problems based on suitable mixed finite element method discretization developed in \cite{boon:nordbotten:yotov}. Beside introducing the mixed-dimensional geometry, the main aspects of the discretization are flux coupling between subdomains using a mortar variable and inf-sup stability of the associated saddle-point problem. Moreover, this framework has been shown to be well incorporated within functional analysis as a concept of mixed-dimensional partial differential equations \cite{boon:nordbotten:vatne}, allowing even further applications in poroelasticity and transport problems. 
	
	We propose a set of block preconditioners for Krylov subspace methods for solving the linear system of equations arising from the chosen discretization. Following the theory in \cite{mardal:winther} and \cite{loghin:wathen}, we derive uniform block preconditioners based on the well-posedness of an alternative but equivalent formulation.  Proper weighted norm is chosen so that the well-posedness constants are robust with respect to the physical and discretization parameters but depend on the shape regularity of the meshes. Both block diagonal and triangular preconditioners are developed based on the framework~\cite{loghin:wathen,mardal:winther}.  Those block preconditioners are not only theoretically robust and effective but can also be implemented straightforwardly by taking advantage of the block structure of the problem.  
	
	
	The rest of the paper is organized as follows. In Section~\ref{sec:prelim} we first introduce the mixed-dimensional geometry and the governing equations of the single-phase flow in fractured porous media followed by the variational formulation and the stable mixed finite element discretization of the problem. The framework of the block preconditioners is briefly recalled in Section~\ref{sec:precond} and its application to mixed-dimensional discretization of flow in fractured porous media is proposed and analyzed in Section~\ref{sec:precond_MD}. We verify the theoretical results by testing several numerical examples in Section~\ref{sec:results} and finalize the paper with concluding remarks in Section~\ref{sec:conclusion}.

	
	\section{Preliminaries}
	\label{sec:prelim}

	In this section, we set up the problem of flow in fractured porous media following \cite{boon:nordbotten:yotov}. Let $ \Omega^n $ be a domain of the porous medium of dimension $ n=2, 3 $ that can be decomposed by fractures into $ \Omega_i^n, i \in I^n $. The fractures and their intersections are represented as lower $ d $-dimensional manifolds $ \Omega_i^d, i \in I^d $, $0 \leq d < n$, and inherit the similar decomposition structure as the porous medium $ \Omega^n $ (see Figure \ref{sec:2.fig:geometrical_representation}). Here, we use $ I^d $ as a local index set in dimension $ 0 \leq d \leq n $. Furthermore, we define $ \Gamma^d_{ij} $ for $ j \in J^d_i \subseteq I^d $ as interfaces between $ \Omega_i^{d+1} $ and adjacent $ \Omega_j^{d} $. Union over the subscript set $ I^d $ represents all $ d $-dimensional subdomains, that is 
	\begin{align}
		\Omega^d & = \bigcup\limits_{i \in I^d} \Omega^d_i, \\
		\Gamma^d & = \bigcup\limits_{i \in I^d} \Gamma^d_i = \bigcup\limits_{i \in I^d} \bigcup\limits_{j\in J^d_i} \Gamma^d_{ij}.
	\end{align}
	Finally, the fractured porous medium domain $ \Omega $ with interface $ \Gamma $ is defined as
	\begin{align}
		\Omega &= \bigcup\limits_{d=0}^{n} \Omega^d, \qquad \Gamma = \bigcup\limits_{d=0}^{n-1} \Gamma^d.
	\end{align}
	
	\begin{figure}[htbp]
		\centering \includegraphics[width=0.4\textwidth]{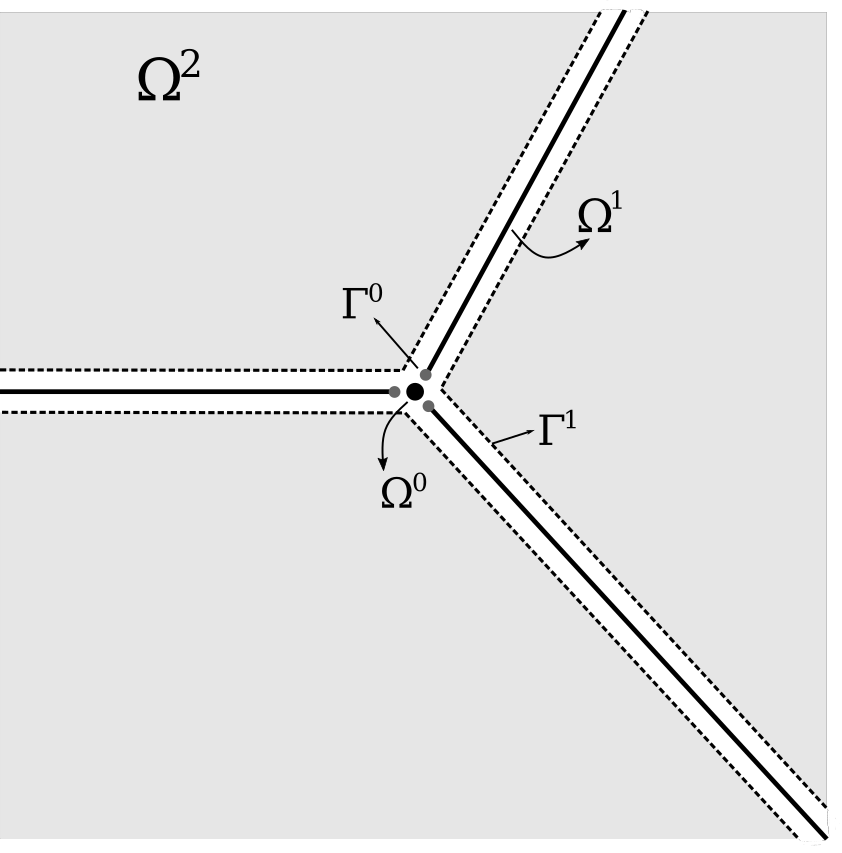}
		\hspace{1cm}
		\includegraphics[width=0.5\textwidth]{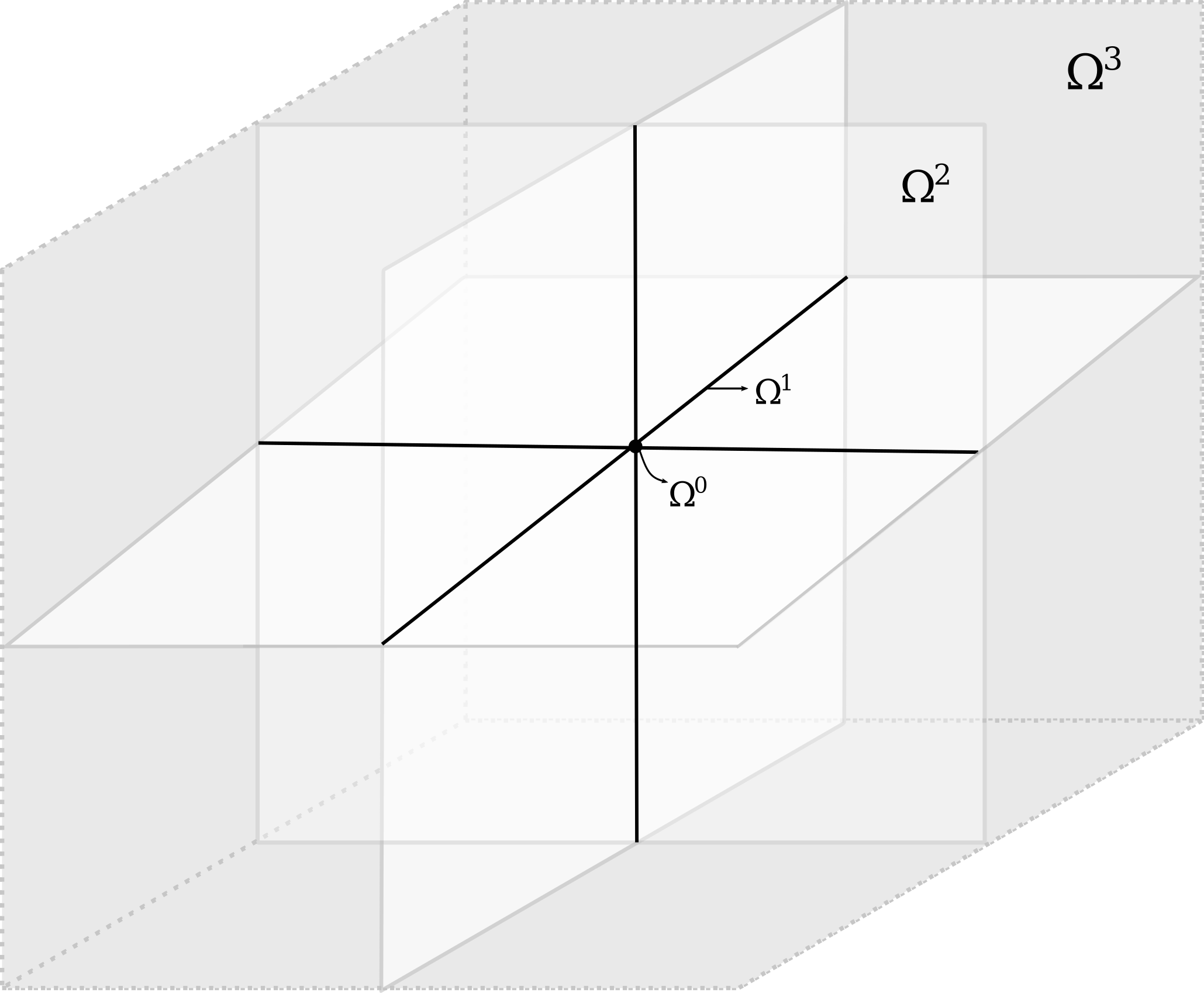}
		\caption{An illustration of the dimensional decomposition of the domain of the fractured porous media, in two (left) and three (right) dimensions. The dimension of each subdomain $ \Omega^d $ is given in the superscript $ d $. In the case of intersecting fractures, $ \Gamma^d $ is set as a union of interfaces adjacent to all subdomains $ \Omega^d $.}
		\label{sec:2.fig:geometrical_representation}
	\end{figure}
	
	\begin{remark}
		\label{sec:2.rem:geometry}
		Even though the theoretical results in \cite{boon:nordbotten:yotov, boon:nordbotten:vatne} allow for a more complex geometrical structure, for the sake of simplicity we restrict the model to domains of rectangular type. That is, we approximate fractures as lines on a plane for $ n = 2 $ or flat surfaces in a box for $ n = 3 $. However, we allow for any configuration of fractures or fracture intersections within, for example, very acute angles of fracture intersections, multiple intersecting fractures or T-type intersections.
	\end{remark}

	Now that we have set up the dimensional decomposition framework for the fractured porous medium, we introduce the governing laws in the subdomains and fractures. First, notation and properties of the physical parameters are introduced. For the sake of simplicity, we slightly abuse the notation by omitting subdomain subscripts and dimension superscripts in the following definitions. We only keep the indices in certain cases when clarification is necessary. 
	
	Assume that the boundary of $ \Omega $ can be partitioned to $ \partial\Omega = \partial\Omega_D \cup \partial\Omega_N $ such that $ \partial\Omega_D \cap \partial\Omega_N = \emptyset $ and $ \partial\Omega_D $ is of positive measure. We adopt the notation in each dimension $ 0 < d \leq n $ , that is
	\begin{align}
	\partial \Omega^d_{iD} & = \partial\Omega^d_i \cap \partial\Omega_D, \qquad \partial \Omega^d_{iN} = \partial\Omega^d_i \cap \partial\Omega_N.
	\end{align}
	
	The material permeability $ K $ and normal permeability $ K_{\bm{\nu}} $ tensors are considered to be bounded both above and below, symmetric and positive definite, and we denote $ \bm{\nu} $ as an outward unit normal on $ \Omega $.
	Furthermore, let $ \gamma^d_{ij} $ be the distance from $ \Gamma^d_{ij} $ to $ \Omega^d_i $, which for $ d = n - 1 $ represents the fracture aperture.   
	The physical parameters $ K $ and $ \gamma $ may vary spatially. However, to simplify the analysis, we assume that they are constant on each subdomain in each dimension.
	
	In each $ \Omega^d $, we introduce the governing Darcy's law and mass conservation, find fluid velocity $ \bm{u}^d $ and pressure $ p^d $ that satisfy
	\begin{subequations}
		\label{sec:2.eq:ge}
		\begin{align}
			\bm{u}^d & = - K\nabla p^d, & \text{ in } \Omega^d, \quad & 0 \leq d  \leq n, \label{sec:2.eq:ge_a} \\ 
			\nabla \cdot \bm{u}^d + \llbracket \lambda^d \rrbracket & = f^d, & \text{ in } \Omega^d, \quad  & 0 \leq d  \leq n,	\label{sec:2.eq:ge_b} 
		\end{align}
	\end{subequations}
	where we introduce an additional mortar variable $\lambda^d$, defined as 
	\begin{align}
	\label{sec:2.eq:lambda}
	\lambda^d|_{\Gamma^d_{ij}}=\lambda^{d}_{ij} &= \bm{u}^d \cdot \bm{\nu}, & \text{ on } \Gamma^{d}_{ij}, \qquad & j \in  J^{d}_{i}, \text{ } i \in I^d, \text{ } 0 \leq d  \leq n-1,
	\end{align}
	to account for the mass transfer across each interface $\Gamma^d_{ij}$,
	and a jump operator $ \llbracket \cdot \rrbracket : L^2(\Gamma^{d}) \to L^2(\Omega^{d}) $ as
	\begin{equation}
	\label{sec:2.eq:jump}
		\llbracket \lambda^d \rrbracket|_{\Omega^{d}_i} = - \sum\limits_{j \in J^{d}_{i}} \lambda^d_{ij}, \qquad  i \in I^d, \quad 0 \leq d \leq n.
	\end{equation}
	Since there is no notion of interface $ \Gamma^n $ or flow in a point $ \Omega^0 $, we extend the definion of $ \lambda^n $ and $ \bm{u}^0 $ by setting them equal to zero.
	
	An additional interface law on $ \Gamma^d_{ij} $ is introduced to describe the normal flow due to the difference in pressure from $ \Omega^d_i $ to $ \Omega^{d+1} $, 
	\begin{align}
	\label{sec:2.eq:normal_darcy}
		\lambda^d_{ij} & = -K_{\bm{\nu}} \dfrac{p^d_i - p^{d+1}|_{\Gamma^d_{ij}}}{\gamma_{ij}^d}, & \text{ on } \Gamma^{d}_{ij}, \qquad & j \in  J^{d}_{i}, \text{ } i \in I^d, \text{ } 0 \leq d  \leq n-1.
	\end{align}
	
	Finally, proper boundary conditions are needed. For example, 
	\begin{align}
		p^d & = g^d, & \text{ on } \partial\Omega^d_D, \quad  & 0 \leq d  \leq n,	\label{sec:2.eq:bc_a}\\
		\bm{u}^d \cdot \bm{\nu} & = 0, & \text{ on } \partial\Omega^d_N, \quad & 0 \leq d  \leq n.	\label{sec:2.eq:bc_b}
	\end{align}

	\begin{remark}
		\label{sec:2.rem:scaling}
		In the previous equations, we have used $ \bm{u}^d $ as integrated flux and $ p^d $ as averaged pressure in each $ \Omega^d $, $0 \leq d \leq n$. Therefore, the scaling with the cross-sectional area $ \varepsilon $ of order $ \mathcal{O}(\gamma^{n-d})$ due to the model reduction has been accounted for within the permeability parameters $ K $ and $ K_{\bm{\nu}} $.
	\end{remark}
	
		
	\subsection{Variational formulation}
	\label{sec:2.subsec:var_form}
	
	Now we consider the variational form of the problem \eqref{sec:2.eq:ge}--\eqref{sec:2.eq:bc_b}. For any open bounded set $ \omega \in \mathbb{R}^n $, let $ L^2(\omega) $ and $ H^s(\omega) $ denote the $L^2$ space and the standard Sobolev spaces on functions defined on $ \omega $, respectively. Also, denote $ H^{\frac{1}{2}}(\partial \omega) $ as the space of $ L^2 $-traces on the boundary $ \partial \omega $ of functions in $ H^1(\omega) $. Let $ ( \cdot, \cdot )_{\omega} $ be the $ L^2 $-inner product and $ \| \cdot \|_{L^2(\omega)} $ the induced $ L^2 $-norm.  We define
	\begin{align*}
	\label{sec:2.eq:spaces}
		\bm{V}^d & = \{ \bm{v} \in (L^2(\Omega^d))^d : \nabla \cdot \bm{v} \in L^2(\Omega^d), (\bm{v} \cdot \bm{\nu})|_{ \partial \Omega^d_N } = 0  \}, & 1 \leq d \leq n, \\
		\Lambda^d & = L^2(\Gamma^d), & 0 \leq d \leq n-1, \\
		Q^d & = L^2(\Omega^d), & 0 \leq d \leq n,
	\end{align*}
	where $ \bm{V}^d $ representing the flux function space on $ \Omega^d $, $ Q^d $ the pressure space on $ \Omega^d $, and $ \Lambda^d $ the function space of normal flux across interface $ \Gamma^d $. Furthermore, let $ \bm{V}^d_0 $ be a subspace of $ \bm{V}^d $ containing functions $ \bm{v}_0 $ such that $ \bm{v}_0 \cdot \bm{\nu} = 0$ on $ \Gamma^{d-1} $. In addition, define the extension operator $ R^{d} : \Lambda^{d} \to \bm{V}^{d+1} $ as
	\begin{equation}
	\label{sec:2.eq:extension_op}
		R^{d} \lambda^d \cdot \bm{\nu} = 
		\begin{cases}
			\lambda^d, & \text{ on } \Gamma^d, \\
			0, & \text{ elsewhere}.
		\end{cases}
	\end{equation}
	To summarize the formulation, we compose function spaces over dimensions
	\begin{equation}
	\label{sec:2.eq:composite_spaces}
		\bm{V} = \bigoplus\limits^{n}_{d=1} \bm{V}^d, \quad \bm{V}_0 = \bigoplus\limits^{n}_{d=1} \bm{V^d}_0, \quad \Lambda = \bigoplus\limits^{n-1}_{d=0} \Lambda^d,	\quad Q = \bigoplus\limits^{n}_{d=0} Q^d, 
	\end{equation}
	and associate composite $ L^2 $-inner products
	$$
	( \cdot, \cdot )_{\Omega} = \sum\limits^{n}_{d=0}( \cdot, \cdot )_{\Omega^d} = \sum\limits^{n}_{d=0} \sum\limits_{i \in I^d} ( \cdot, \cdot )_{\Omega^d_i},
	\quad
	( \cdot, \cdot )_{\Gamma} = \sum\limits^{n-1}_{d=0}( \cdot, \cdot )_{\Gamma^d} = \sum\limits^{n-1}_{d=0} \sum\limits_{i \in I^d} \sum\limits_{j \in J^d_i}( \cdot, \cdot )_{\Gamma^d_{ij}}.
	$$
	and induced composite $ L^2 $-norms
	$$
	\|\cdot \|^2_{ L^2(\Omega) } = \sum\limits^{n}_{d=0}\|\cdot \|^2_{ L^2(\Omega^d) }, \quad \|\cdot \|^2_{ L^2(\Gamma) } = \sum\limits^{n}_{d=0}\|\cdot \|^2_{ L^2(\Gamma^d) }
	$$
	Finally, let $ R: \Lambda \to \bm{V} $ be defined as $ R = \bigoplus\limits^{n-1}_{d=0} R^d $.
	
	The system \eqref{sec:2.eq:ge}--\eqref{sec:2.eq:bc_b} in the weak formulation reads: Find $ (\bm{u}_0, \lambda, p) \in  \bm{V}_0 \times \Lambda \times Q $ that satisfies	
	\begin{subequations}
		\label{sec:2.eq:vf}
		\begin{align}
		\left( K^{-1}(\bm{u}_0 + R\lambda), \bm{v}_0 \right)_{\Omega} -  \left(p, \nabla \cdot \bm{v}_0 \right)_{\Omega} ={} & -\left( g, \bm{v}_0 \cdot \bm{\nu} \right)_{\partial \Omega_D}, & \forall \, \bm{v}_0 \in \bm{V}_0, \label{sec:2.eq:vf_a} \\
		\left( K^{-1}(\bm{u}_0 + R\lambda), R\mu \right)_{\Omega} - \left( p, \nabla \cdot R\mu \right)_{\Omega} \nonumber \\
		+ \left( \gamma K_{\bm{\nu}}^{-1}  \lambda, \mu \right)_{\Gamma}
		-  (p, \llbracket \lambda \rrbracket)_{\Omega} ={} & 0, &  \forall \, \mu \in \Lambda, \label{sec:2.eq:vf_b} \\
		-( \nabla \cdot(\bm{u}_0 + R\lambda), q)_{\Omega}  -  (\llbracket \lambda \rrbracket, q)_{\Omega} ={} & -( f, q )_{\Omega}, & \forall \, q \in Q, \label{sec:2.eq:vf_c}
		\end{align}
	\end{subequations}
	with $ g \in H^{\frac{1}{2}}(\partial \Omega_D) $ and $ f \in L^2(\Omega) $. As before, functions $ \bm{u}^0_0, \bm{v}^0_0, \lambda^n $ and $\mu^n $ are set to zero.
	
	We end this section by observing the saddle point structure of the system \eqref{sec:2.eq:vf}. First, let $ \bm{W} = \bm{V}_0 \times \Lambda $ be the function space of all flux variables, including mortar variable, and define the mixed-dimensional divergence operator $ \bm{D} \cdot : \bm{W} \to Q$ as
	\begin{equation}
		\label{sec:2.eq:div}
		\bm{D} \cdot \bm{w} = \bm{D} \cdot [\bm{u}_0, \lambda] = \nabla \cdot \bm{u}_0 + \jump{\lambda}, \qquad \bm{w} \in \bm{W}.
	\end{equation} 	
	Define the following two bilinear forms
	\begin{subequations}
		\label{sec:2.eq:bilinear_forms}
		\begin{align}
		a(\bm{w}, \bm{r}) & = \left( K^{-1}(\bm{u}_0 + R\lambda), \bm{v}_0 +  R\mu \right)_{\Omega} + \left( \gamma K_{\bm{\nu}}^{-1}  \lambda, \mu \right)_{\Gamma}, \label{sec:2.eq:bilinear_forms_a} \\
		b(\bm{r}, p) & = -\left(p, \bm{D} \cdot [\bm{v}_0 + R\mu, \mu] \right)_{\Omega}. \label{sec:2.eq:bilinear_forms_b}
		\end{align}
	\end{subequations}
    Then the saddle point form of system \eqref{sec:2.eq:vf} reads: Find $ (\bm{w}, p) \in  \bm{W} \times Q $ such that
    \begin{subequations}
    	\label{sec:2.eq:saddle_point}
    	\begin{align}
    	a(\bm{w},\bm{r}) + b(\bm{r}, p) &= -(g, \bm{v}_0 \cdot  \bm{\nu})_{\partial \Omega_D}, & \forall \, \bm{r} \in \bm{W},  \\
    	b(\bm{w},q) \hskip 40pt &= -(f,q)_{\Omega}, & \forall \, q \in Q.
    	\end{align}
    \end{subequations}
   It has been shown in~\cite{boon:nordbotten:yotov} that the bilinear forms $ a(\cdot, \cdot) $ and $ b(\cdot, \cdot) $ are continuous with respect to the following norms for $\bm{r} = [\bm{v}_0, \mu] \in \bm{W}$ and $q \in Q$,
   \begin{subequations}
   	\label{sec:2.eq:norms}
   	\begin{align}
   	\| \bm{r} \|_{\bm{W}}^2 &  = \| K^{-\frac{1}{2}} (\bm{v}_0 + R\mu)  \|_{L^2(\Omega)}^2 + \| \gamma^{\frac{1}{2}} K_{\bm{\nu}}^{-\frac{1}{2}} \mu \|_{L^2(\Gamma)}^2 \nonumber \\
   	& \quad + \| \bm{D} \cdot [\bm{v}_0 + R\mu, \mu] \|_{L^2(\Omega)}^2, \label{sec:2.eq:norms_w} \\
   	\| q \|_Q^2 & = \| q \|_{L^2(\Omega)}^2. \label{sec:2.eq:norms_q}
   	\end{align}
   \end{subequations}
    In addition, $a(\cdot, \cdot)$ is shown to be coercive on the kernel of $b(\cdot, \cdot)$ in~\cite{boon:nordbotten:yotov} as well.  Finally, the following theorem states that $b(\cdot,\cdot)$ satisfies the inf-sup condition.
    	\begin{citedthm}[\cite{boon:nordbotten:yotov}]
    	\label{sec:2.th:infsup}
    	Let the bilinear form $ b(\cdot, \cdot) $ be defined as in \eqref{sec:2.eq:bilinear_forms_b}. Then there exists a constant $ \beta > 0 $ independent of the physical parameters $ K $, $ K_{\bm{\nu}} $ and $ \gamma $ such that
    	\begin{equation}
    	\inf\limits_{q \in Q} \sup\limits_{\bm{r} \in \bm{W}}
    	\dfrac{b(\bm{r}, q)}{\| \bm{r} \|_{\bm{W}} \| q \|_Q} \geq \beta.
    	\end{equation}
    \end{citedthm}
	Following the classical Brezzi theory~\cite{Brezzi1974,BoffiBrezziFortin2013}, we conclude that the saddle point system \eqref{sec:2.eq:saddle_point} is well-posed, i.e., there exists a unique solution of \eqref{sec:2.eq:saddle_point}.

	\subsection{Discretization}
	\label{sec:2.subsec:discr}
	
	We continue this section with discretizing the problem \eqref{sec:2.eq:saddle_point} by the mixed finite element approximation.
	Let $ \mathcal{T}^d_{\Omega} $ and $ \mathcal{T}^d_{\Gamma} $ denote a d-dimensional shape-regular triangulation of $ \Omega^d $ and $ \Gamma^d $, and $ h = \max\limits_{0 \leq d \leq n} h^d  $ the characteristic mesh size parameter. Consider $ \bm{V}^d_h \subset \bm{V}^d $, $ \bm{V}^d_{0h} \subset \bm{V}^d_0 $, $ Q^d_h  \subset Q^d $ and $ \Lambda^d_h \subset \Lambda^d $ to be the lowest-order stable mixed finite element approximations on subdomain mesh $ \mathcal{T}^d_{\Omega} $ and mortar mesh $ \mathcal{T}^d_{\Gamma} $. That is, $ \bm{V}^d_h = \mathbb{RT}_0(\mathcal{T}^d_{\Omega}) $, $ \Lambda^d_h = \mathbb{P}_0(\mathcal{T}^d_{\Gamma})$ and $ Q^d_h = \mathbb{P}_0(\mathcal{T}^d_{\Omega}) $, where $ \mathbb{RT}_0 $ stands for lowest-order Raviart-Thomas(-N\'{e}d\'{e}lec) spaces \cite{nedelec, raviart:thomas} and $ \mathbb{P}_0 $ for the space of piecewise constants. Furthermore, define $ \widehat{\Pi}^d_h : \Lambda^d_h \to \bm{V}^{d+1}_h \cdot \bm{\nu}|_{\Gamma^d} $ 
		to be the $ L^2 $-projection operator such that,  for any $ \mu^d_h \in \Lambda^d_h $, 
		\begin{align}
		(  \widehat{\Pi}^d_h \mu^d_h - \mu^d_h, \bm{v} \cdot \bm{\nu} )_{L^2(\Gamma^d)} & = 0, \quad &  \forall \bm{v} \in \bm{V}^{d+1}_h. \label{sec:2.eq:projection_hat}
		\end{align}
	Then we can define the discrete extension operator $ R^d_h : \Lambda^d_h \to \bm{V}^d_h $,
	\begin{equation}
	\label{sec:2.eq:extension_op_discrete}
		R^{d}_h \lambda^d \cdot \bm{\nu} = 
		\begin{cases}
			\widehat{\Pi}^d_h \lambda^d, & \text{ on } \Gamma^d, \\
			0, & \text{ elsewhere}.
		\end{cases}
	\end{equation}
	Analogous to the continuous case, we define the discrete composite spaces
	\begin{equation}
	\label{sec:2.eq:composite_spaces_discrete}
		\bm{V}_h = \bigoplus\limits^{n}_{d=1} \bm{V}^d_h, \quad \bm{V}_{0h} = \bigoplus\limits^{n}_{d=1} \bm{V}^d_{0h}, \quad \Lambda_h = \bigoplus\limits^{n-1}_{d=0} \Lambda^d_h, \quad Q_h = \bigoplus\limits^{n}_{d=0} Q^d_h,
	\end{equation}
	and the linear operators $ \widehat{\Pi}_h = \bigoplus\limits^{n-1}_{d=0} \widehat{\Pi}^d_h $ and $ R_h = \bigoplus\limits^{n-1}_{d=0} R^d_h $.
%

	With $ \bm{W}_h = \bm{V}_{0h} \times \Lambda_h $, the finite element approximation of the system \eqref{sec:2.eq:vf} is formulated as follows: Find $ (\bm{w}_h, p_h) \in  \bm{W}_h \times Q_h $ such that,
	\begin{subequations}
		\label{sec:2.eq:saddle_point_h}
		\begin{align}
		a(\bm{w}_h,\bm{r}_h) + b(\bm{r}_h, p_h) &= -(g, \bm{v}_{0h} \cdot  \bm{\nu})_{\partial \Omega_D}, & \forall \, \bm{r}_h \in \bm{W}_h,  \\
		b(\bm{w}_h,q_h) \hskip 50pt &= -(f,q_h)_{\Omega}, & \forall \, q_h \in Q_h.
		\end{align}
	\end{subequations}

	Due to our choice of the finite element spaces, the continuity of $a(\cdot,\cdot)$ and $b(\cdot,\cdot)$ and the coercivity of $a(\cdot,\cdot)$ on the kernel of $b(\cdot,\cdot)$ are preserved naturally.  To show the well-posedness of the discrete saddle point system~\eqref{sec:2.eq:saddle_point_h}, we need the inf-sup condition to hold on the discrete spaces as well.  This has been shown in~\cite{boon:nordbotten:yotov} and is stated in the following theorem. 
	\begin{citedthm}[\cite{boon:nordbotten:yotov}]
		\label{sec:2.th:infsup_h}
		There exists a constant $ \beta > 0 $ independent of the discretization parameter $ h $ and the physical parameters $ K $, $ K_{\bm{\nu}} $ and $ \gamma $ such that
		\begin{equation}
			\label{sec:2.eq:infsup_h}
			\inf\limits_{q_h \in Q_h} \sup\limits_{\bm{r}_h \in \bm{W}_h}
			\dfrac{b(\bm{r}_h, q_h)}{\| \bm{r}_h \|_{\bm{W}_h} \| q_h \|_{Q_h}} \geq \beta.
		\end{equation}
	\end{citedthm}
	Therefore, the finite element method~\eqref{sec:2.eq:saddle_point_h} is well-posed by the Brezzi theory~\cite{Brezzi1974,BoffiBrezziFortin2013}.
	
	We finalize this section with the block formulation of the discrete saddle point system \eqref{sec:2.eq:saddle_point_h}. Let linear operators $A: \bm{W}_h \to \bm{W}_h' $ and $B: \bm{W}_h \to Q_h'$ be defined as $ \langle A \bm{w}_h, \bm{r}_h \rangle = a(\bm{w}_h, \bm{r}_h) $ and $ \langle B \bm{r}_h, p_h \rangle = b(\bm{r}_h, p_h)  $, respectively. Here $\bm{W}_h'$ and $Q_h'$ denote the dual spaces of $\bm{W}_h$ and $Q_h$, respectively, and $\langle \cdot, \cdot \rangle$ denotes the duality pairing.  Then \eqref{sec:2.eq:saddle_point_h} is equivalent to the following operator form,
	\begin{equation}
		\label{sec:2.eq:block_form_h}
		\mathcal{A}
		\begin{pmatrix}
			\bm{w}_h \\ p_h
		\end{pmatrix}
		=
		\begin{pmatrix}
		G \\ F
		\end{pmatrix}
		\quad \text{ with } 
		\mathcal{A} = 
		\begin{pmatrix}
		 A & B^T \\
		 -B & 0
		\end{pmatrix},
	\end{equation}	
	with $ G([\bm{v}_{0h}, \lambda_h]) \coloneqq - ( g, \bm{v}_{0h} \cdot \bm{\nu} )_{\partial \Omega_D} $ and $ F(q_h) \coloneqq ( f, q_h )_{\Omega} $. 
	
	The well-posedness of the system~\eqref{sec:2.eq:saddle_point_h} ensures that $ \mathcal{A} $ is an isomorphism from $ \bm{W}_h \times Q_h $ to its dual $\bm{W}_h' \times Q_h'$ and, therefore, \eqref{sec:2.eq:block_form_h} has a unique solution $ (\bm{w}_h, p_h) \in \bm{W}_h \times Q_h $.

	\section{Block preconditioners}
\label{sec:precond}


	In this section, we briefly present the general preconditioning theory for designing block preconditioners of Krylov subspace iterative
	methods~\cite{loghin:wathen,mardal:winther}, which introduces necessary tools for the analysis in the following section. 

	The block preconditioning framework~\cite{loghin:wathen,mardal:winther} is based on the well-posedness theory.  Therefore, we first introduce the setup of the problem.  Let $ \bm{X} $ be a real separable Hilbert space and $ (\cdot, \cdot)_{\bm{X}} $ represent the inner product on $ \bm{X} $ that induces the norm $ \| \cdot \|_{\bm{X}} $.
	Furthermore, denote $ \bm{X}' $ as a dual space to $ \bm{X} $ and $ \langle \cdot, \cdot \rangle $ as the duality pairing between them. Let $ \mathcal{L}(\cdot, \cdot) $ be a bilinear form on $ \bm{X} $ that satisfies the continuity condition and the inf-sup condition,
	\begin{equation}
		\label{sec:3.eq:bilin_form}
		\inf\limits_{\bm{x} \in \bm{X}} \sup\limits_{\bm{y} \in \bm{X}}
		\dfrac{ \mathcal{L}(\bm{x}, \bm{y}) }{ \| \bm{x} \|_{\bm{X}} \| \bm{y} \|_{\bm{X}} } \geq \beta \quad \text{and} \quad 
		| \mathcal{L}(\bm{x}, \bm{y}) | \leq \alpha \| \bm{x} \|_{\bm{X}} \| \bm{y} \|_{\bm{X}}, \; \forall \bm{x}, \bm{y} \in \bm{X},
	\end{equation}
	for $ \alpha, \beta > 0 $. We aim to construct a robust preconditioner for the linear system
	\begin{equation}
		\label{sec:3.eq:lin_sys}
		\mathcal{A} \bm{x} = \bm{b},
	\end{equation}
	where $ \mathcal{A} : \bm{X} \to \bm{X}' $ is induced by the bilinear form $ \mathcal{L}(\cdot, \cdot) $ such that $ \langle \mathcal{A} \bm{x}, \bm{y} \rangle = \mathcal{L}(\bm{x}, \bm{y}) $. The properties of the bilinear form ensure that $ \mathcal{A} $ is a bounded and symmetric linear operator and the system~\eqref{sec:3.eq:lin_sys} is well-posed.  Our goal is to develop block preconditioners for solving~\eqref{sec:3.eq:lin_sys}.


\subsection{Norm-equivalent Preconditioner}
\label{sec:3.subsec:norm-equiv}

	Consider a symmetric positive definite operator $\mathcal{M}: \bm{X}' \to \bm{X}$ which induces an inner product $(\bm{x},\bm{y})_{\mathcal{M}^{-1}}:=\langle \mathcal{M}^{-1}\bm{x},\bm{y}\rangle$ on $\bm{X}$ and corresponding norm $\| \bm{x} \|^2_{\mathcal{M}^{-1}}:=(\bm{x},\bm{x})_{\mathcal{M}^{-1}}$. Naturally, $\mathcal{MA}:\bm{X} \to \bm{X}$ is symmetric with respect to $(\cdot,\cdot)_{\mathcal{M}^{-1}}$ and we can use $\mathcal{M}$ as a preconditioner for the MINRES algorithm whose convergence rate is stated in the following theorem.
	\begin{theorem}[\cite{greenbaum1997}]
		Let $\bm{x}^m$ be the $m$-th iteration of the MINRES method preconditioned with $\mathcal{M}$ and $\bm{x}$ be the exact solution, it follows that
		\begin{equation*}
			\| \mathcal{A} (\bm{x} - \bm{x}^m) \|_{\mathcal{M}} \leq 2 \rho^m \| \mathcal{A}(\bm{x} - \bm{x}^0) \|_{\mathcal{M}},
		\end{equation*}
		where $\rho = \frac{\kappa(\mathcal{MA})-1}{\kappa(\mathcal{MA})+1}$ and $\kappa(\mathcal{MA})$ denotes the condition number of $\mathcal{MA}$. 
	\end{theorem}
	As shown in~\cite{mardal:winther}, if \eqref{sec:3.eq:bilin_form} holds and $\mathcal{M}$ satisfies,
	\begin{equation}
		\label{sec:3.eq:norm-equiv}
		c_1 \| \bm{x} \|^2_{\bm{X}} \leq \| \bm{x} \|^2_{\mathcal{M}^{-1}} \leq c_2 \| \bm{x} \|^2_{\bm{X}},
	\end{equation}
	then $\mathcal{A}$ and $\mathcal{M}$ are called \emph{norm-equivalent} and $\kappa(\mathcal{MA}) \leq \frac{c_2 \alpha }{c_1 \beta }$.  Thus, if the well-posedness constants $\alpha$ and $\beta$ and the norm-equivalence constants $c_1$ and $c_2$ are all independent of the physical and discretization parameters, then $\mathcal{M}$ provides a robust preconditioner.  

	One natural choice of the norm-equivalent preconditioner is the \emph{Riesz operator} $ \mathcal{B} : \bm{X}' \to \bm{X} $ corresponding to the inner product $ (\cdot, \cdot)_{\bm{X}} $
	\begin{equation}
		\label{sec:3.eq:riesz_map}
		(\mathcal{B} \bm{f}, \bm{x})_{\bm{X}} =  \langle \bm{f}, \bm{x} \rangle, \qquad \forall \bm{f} \in \bm{X}', \; \bm{x} \in \bm{X}.
	\end{equation}	
	It is easy to see that if we choose $\mathcal{M} = \mathcal{B}$, then~\eqref{sec:3.eq:norm-equiv} holds with constants $c_1=c_2=1$ and, therefore,   
	the preconditioned system 
	\begin{equation}
		\label{sec:3.eq:precond_sys}
		\mathcal{B} \mathcal{A} \bm{x} = \mathcal{B} \bm{b}
	\end{equation}
	has a bounded condition number
	\begin{equation}
		\label{sec:3.eq:cond_no}
		\kappa(\mathcal{B} \mathcal{A}) = \| \mathcal{B} \mathcal{A} \|_{\mathscr{L}(\bm{X}, \bm{X})} \| (\mathcal{B} \mathcal{A})^{-1} \|_{\mathscr{L}(\bm{X}, \bm{X})} \leq \dfrac{\alpha}{\beta}.
	\end{equation}
	If the constants $ \alpha $ and $ \beta $ are independent of the discretization and physical parameters, we obtain a robust preconditioner. 


\subsection{Field-of-values-equivalent Preconditioner}
	\label{sec:3.subsec:FOV-equiv}
	In this section, we recall the class of field-of-values-equivalent (FOV-equivalent) preconditioners which allow more general preconditioners than the norm-equivalent ones.  

	Consider a general operator $\mathcal{M}_L:\bm{X}' \to \bm{X}$ which can be used as a preconditioner for the GMRES method.  The following theorem, developed in \cite{eisenstat:elman:schultz, elman:1982}, characterizes the convergence rate of the GMRES method.
	\begin{theorem}[\cite{eisenstat:elman:schultz, elman:1982}]
		Let $\bm{x}^m$ be the $m$-th iteration of the GMRES method preconditioner with $\mathcal{M}_L$ and $\bm{x}$ be the exact solution, it follows that
		\begin{equation*}
			\| \mathcal{M}_L\mathcal{A}(\bm{x}-\bm{x}^m)\|^2_{\mathcal{M}^{-1}} \leq \left( 1 - \frac{\Sigma^2}{\Upsilon^2}\right) \| \mathcal{M}_L \mathcal{A} (\bm{x} - \bm{x}^0) \|^2_{\mathcal{M}^{-1}},
		\end{equation*}
		where, for any $\bm{x} \in \bm{X}$,
		\begin{equation*}
			\Sigma \leq \frac{(\mathcal{M}_L\mathcal{A}\bm{x},\bm{x})_{\mathcal{M}^{-1}}}{(\bm{x},\bm{x})_{\mathcal{M}^{-1}}}, \qquad \frac{\| \mathcal{M}_L \mathcal{A} \bm{x}\|_{\mathcal{M}^{-1}}}{\| \bm{x}\|_{\mathcal{M}^{-1}}} \leq \Upsilon. 
		\end{equation*}
	\end{theorem}
	$\mathcal{M}_L$ is referred to as an FOV-equivalent preconditioner if the constants $\Sigma$ and $\Upsilon$ are independent of the physical and discretization parameters.  Usually $\mathcal{M}_L$ provides a uniform left preconditioner for GMRES.   

	In a similar manner, we can introduce a right preconditioner for GMRES, $\mathcal{M}_U:\bm{X}' \to \bm{X}$ and consider the preconditioned system
	\begin{equation*}
		\mathcal{A} \mathcal{M}_U \bm{y} = \bm{b}, \qquad \bm{x} = \mathcal{M}_U \bm{y}.
	\end{equation*}
	By introducing an inner product on $\bm{X}'$, defined as $(\bm{x}', \bm{y}')_{\mathcal{M}}:= \langle \bm{x}', \mathcal{M}\bm{y'} \rangle$, we say $\mathcal{M}_U$ and $\mathcal{A}$ are FOV-equivalent if, for any $\bm{x}' \in \bm{X}'$, 
	\begin{equation*}
		\Sigma \leq \frac{(\mathcal{A}\mathcal{M}_U\bm{x}',\bm{x}')_{\mathcal{M}}}{(\bm{x}',\bm{x}')_{\mathcal{M}}}, \qquad \frac{\| \mathcal{A}\mathcal{M}_U \bm{x}'\|_{\mathcal{M}}}{\| \bm{x}'\|_{\mathcal{M}}} \leq \Upsilon, 
	\end{equation*}
	where the constants $\Sigma$ and $\Upsilon$ are independent of the physical and discretization parameters. Therefore, $\mathcal{M}_U$ can be used as a uniform right preconditioner for GMRES.

	In many cases~\cite{loghin:wathen,alder:hu:rodrigo:zikatanov,adler:gaspar:hu:ohm:rodrigo:zikatanov}, the FOV-equivalent preconditioners can be derived based on the Riesz operator and the FOV-equivalence can be shown based on the well-posedness conditions \eqref{sec:3.eq:bilin_form}.

	\section{Robust Preconditioners for Mixed-dimensional Model}
	\label{sec:precond_MD}

	In this section, we design block preconditioners based on the general framework mentioned in the previous section.  Consider the finite element approximation \eqref{sec:2.eq:saddle_point_h}.  In this case, define $ \bm{X} = \bm{W}_h \times Q_h $ associated with the following norm 
	\begin{equation}
		\label{sec:4.eq:norm_x}
		\| \bm{y} \|^2_{\bm{X}} =  \| (\bm{r}_h, q_h) \|^2_{\bm{X}} = \| \bm{r}_h \|^2_{\bm{W}} + \| q_h \|^2_{Q}.
	\end{equation}
	Then, the operator $ \mathcal{A} : \bm{X} \to \bm{X}' $ in \eqref{sec:2.eq:block_form_h} is induced by the bilinear form
	\begin{equation}
		\label{sec:3.eq:bilin_A}
		\mathcal{L} (\bm{x}, \bm{y}) = a(\bm{w}_h, \bm{r}_h) +	b(\bm{r}_h, p_h) - b(\bm{w}_h, q_h), 
	\end{equation}
	and satisfies the well-posedness conditions \eqref{sec:3.eq:bilin_form} due to Theorem \ref{sec:2.th:infsup_h}, the continuity of the bilinear forms $ a(\cdot, \cdot) $ and $  b(\cdot, \cdot) $, and the coercivity of $a(\cdot,\cdot)$ on the kernal of $b(\cdot,\cdot)$. Moreover, the constants $ \alpha $ and $ \beta $ are independent of parameters $ h $, $ K $, $ K_{\bm{\nu}} $ and $ \gamma $. 
	
	The Riesz operator corresponding to the norm $ \| \cdot \|_{\bm{X}} $ in \eqref{sec:4.eq:norm_x} is
	\begin{equation}
		\label{sec:4.eq:canonical_precond_blockform}
		\mathcal{B} =
		\begin{pmatrix}
			A + B^T B & 0\\
			0 & I_p
		\end{pmatrix}^{-1},
	\end{equation}
	where $A$ and $B$ are defined as in~\eqref{sec:2.eq:block_form_h} and $I_p$ is the identity operator on $Q$, i.e., $\langle I_p q_h, q_h \rangle = \| q_h \|^2_Q$.
	The main challenge in implementation of this preconditioner is to solve for the upper block $ A + B^T B $ that corresponds to $ I + \grad \, \div $ problem. One way of resolving this is to use auxiliary space theory (see for example \cite{hiptmair:xu, kolev:vassilevski}). However, in our case, additional theory resulting from the mixed-dimensional exterior calculus in \cite{boon:nordbotten:vatne} is needed, which is the topic of our ongoing work~\cite{boon:budisa:hu}. However, in this paper, we consider an alternative formulation of the problem \eqref{sec:2.eq:saddle_point_h} and show the well-posedness with respect to a different weighted norm, which allows for a simpler robust preconditioner.

	\subsection{An Alternative Formulation}
	\label{sec:4.subsec:alter_form}


	In order to introduce the alternative formulation, we need to define a discrete gradient operator $ \bm{D}_h: Q_h \to \bm{W}_h$ such that, for $\bm{r}_h = [\bm{v}_{0h}, \mu_h]$, 
	\begin{equation}
	\label{sec:4.eq:D_h}
		a^D( \bm{D}_h p_h, \bm{r}_h ) = b(\bm{r}_h, p_h) = -\left(p_h, \bm{D} \cdot [\bm{v}_{0h} + R_h\mu_h, \mu_h] \right)_{\Omega},
	\end{equation}
	where, for $\bm{w}_h = [\bm{u}_{0h},\lambda_h]$ and $\bm{r}_h = [\bm{v}_{0h}, \mu_h]$,  
	\begin{equation*}
		a^D(\bm{w}_h, \bm{r}_h):= (K^{-1}(\bm{u}_{0h}+ R_h \lambda_h), \bm{v}_{0h} + R_h \mu_h)_{D, \Omega} + (\gamma K^{-1}_{\bm{\nu}} \lambda_h, \mu_h)_{\Gamma},
	\end{equation*}
	with
	\begin{align*}
		&(K^{-1}(\bm{u}_{0h}+ R_h \lambda_h), \bm{v}_{0h} + R_h \mu_h)_{D,\Omega}:= \\
		& \sum_{d=0}^n \left\{ \sum_{T^d\in \mathcal{T}^d_{\Omega}}\left[ \sum_{f^d \in \partial T^d} ((\bm{u}_{0h}+ R_h \lambda_h)\cdot \bm{\nu}_{f^d}) ((\bm{v}_{0h}+ R_h \mu_h)\cdot \bm{\nu}_{f^d}) (K^{-1} \bm{\phi}_{f^d}, \bm{\phi}_{f^d})_{T^d}  \right]  \right\}.
	\end{align*}
	Here $ T^d \in \mathcal{T}^d_{\Omega} $ is either a tetrahedron for $ d = 3 $, a triangle for $ d = 2 $ or a line segment for $ d = 1 $. Furthermore, $ f^d \in \partial T $ corresponds to a face of the element $ T^d $, $\bm{\nu}_{f^d}$ denotes the unit outer normal of face $f^d$, and $ \bm{\phi}_{f^d} \in \mathbb{RT}_0(T^d) $ is the basis function on face $ f^d $.  Using the discrete gradient operator, an alternative formulation of the system~\eqref{sec:2.eq:saddle_point_h} is given as follows: Find $ (\bm{w}_h, p_h) \in  \bm{W}_h \times Q_h $ such that,
	\begin{subequations}
		\label{sec:4.eq:alter-saddle_point_h}
		\begin{align}
		a(\bm{w}_h,\bm{r}_h) + a^D(\bm{D}_hp_h, \bm{r}_h) &= -(g, \bm{v}_{0h} \cdot  \bm{\nu})_{\partial \Omega_D}, & \forall \, \bm{r}_h \in \bm{W}_h,  \\
		a^D(\bm{D}_hq_h, \bm{w}_h) \hskip 50pt &= -(f,q_h)_{\Omega}, & \forall \, q_h \in Q_h.
	\end{align}
	\end{subequations}
	The well-posedness of the alternative formulation~\eqref{sec:4.eq:alter-saddle_point_h} with respect to the norm \eqref{sec:4.eq:norm_x} follows directly from the well-posedness of the original formulation~\eqref{sec:2.eq:saddle_point_h} because the two formulations are equivalent.  However, in order to derive a block preconditioner different from~\eqref{sec:4.eq:canonical_precond_blockform}, we shall consider the same coefficient operator $\mathcal{A}$~\eqref{sec:2.eq:block_form_h} with a different weak interpretation and the well-posedness in a different setting. 

	The alternative weighted norm we consider for the alternative formulation~\eqref{sec:4.eq:alter-saddle_point_h} is defined as
	\begin{equation} 
	\label{sec:4.eq:alter_norm}
		\vertiii{(\bm{r}_h, q_h)}^2 := \| \bm{r}_h \|_a^2 +  \| \bm{D}_h q_h \|_{a^D}^2,
	\end{equation}
	where $\| \bm{r}_h \|^2_a := a(\bm{r}_h, \bm{r}_h) $ and $\| \bm{r}_h \|^2_{a^D} := a^D(\bm{r}_h, \bm{r}_h)$.  In order to show~\eqref{sec:4.eq:alter-saddle_point_h} (or the operator form~\eqref{sec:2.eq:block_form_h}) is well-posed with respect to this alternative norm~\eqref{sec:4.eq:alter_norm}, we need the following two lemmas.  The first lemma shows that the forms $a(\cdot,\cdot)$ and $a^D(\cdot, \cdot)$ are spectrally equivalent.    

%

	\begin{lemma}
		\label{sec:4.lemma:norm_equi}
		There exist constants $ c_1, c_2 > 0 $, depending only on the shape regularity of the mesh $ \mathcal{T}_{\Omega} $, such that the following inequalities hold,
		\begin{equation}
			\label{sec:4.eq:norm_equi}
			c_1 \| \bm{r}_h \|_{a^D} \leq \| \bm{r}_h \|_{a} \leq c_2 \| \bm{r}_h \|_{a^D}, \qquad \forall \ \bm{r}_h \in \bm{W}_h.
		\end{equation}
	\end{lemma}

	\begin{proof}
		Recall that 
		\begin{align*}
			\| \bm{r}_h \|^2_a & = a( [\bm{v}_{0h}, \mu_h] , [\bm{v}_{0h}, \mu_h] ) \nonumber \\
			& = ( K^{-1} (\bm{v}_{0h} + R_h \mu_h), (\bm{v}_{0h} + R_h \mu_h))_{\Omega} +  (\gamma K_{\bm{\nu}}^{-1} \mu_h, \mu_h)_{\Gamma}, \\ \| \bm{r}_h \|^2_{a^D} & = a^D( [\bm{v}_{0h}, \mu_h] , [\bm{v}_{0h}, \mu_h] ) \nonumber \\
			& = ( K^{-1} (\bm{v}_{0h} + R_h \mu_h), (\bm{v}_{0h} + R_h \mu_h))_{D,\Omega} +  (\gamma K_{\bm{\nu}}^{-1} \mu_h, \mu_h)_{\Gamma}.
		\end{align*}

		Obviously, \eqref{sec:4.eq:norm_equi} holds if $( K^{-1} (\bm{v}_{0h} + R_h \mu_h), (\bm{v}_{0h} + R_h \mu_h))_{\Omega}$ and $( K^{-1} (\bm{v}_{0h} + R_h \mu_h), (\bm{v}_{0h} + R_h \mu_h))_{D,\Omega}$ are spectrally equivalent.  Note that
		\begin{equation*}
			(K^{-1}(\bm{v}_{0h}+ R_h \mu_h), \bm{v}_{0h} + R_h \mu_h)_{\Omega}= 
			\sum_{d=0}^n \sum_{T^d\in \mathcal{T}^d_{\Omega}}(K^{-1}(\bm{v}_{0h}+ R_h \mu_h), \bm{v}_{0h} + R_h \mu_h)_{T^d} ,
		\end{equation*}
		where
		\begin{align*}
			&(K^{-1}(\bm{v}_{0h}+ R_h \mu_h), \bm{v}_{0h} + R_h \mu_h)_{T^d} = \\  
			& \sum_{f^d, \tilde{f}^d \in \partial T^d} ((\bm{v}_{0h}+ R_h \mu_h)\cdot \bm{\nu}_{f^d}) ((\bm{v}_{0h}+ R_h \mu_h)\cdot \bm{\nu}_{\tilde{f}^d}) (K^{-1} \bm{\phi}_{f^d}, \bm{\phi}_{\tilde{f}^d})_{T^d}.
		\end{align*}
		and 
		\begin{equation*}
			(K^{-1}(\bm{v}_{0h}+ R_h \mu_h), \bm{v}_{0h} + R_h \mu_h)_{D,\Omega} = 
			\sum_{d=0}^n  \sum_{T^d\in \mathcal{T}^d_{\Omega}} (K^{-1}(\bm{v}_{0h}+ R_h \mu_h), \bm{v}_{0h} + R_h \mu_h)_{D,T^d} ,
		\end{equation*}
		where
		\begin{align*}
			& (K^{-1}(\bm{v}_{0h}+ R_h \mu_h), \bm{v}_{0h} + R_h \mu_h)_{D,T^d} = \\
			&  \sum_{f^d \in \partial T^d} ((\bm{v}_{0h}+ R_h \mu_h)\cdot \bm{\nu}_{f^d}) ((\bm{v}_{0h}+ R_h \mu_h)\cdot \bm{\nu}_{f^d}) (K^{-1} \bm{\phi}_{f^d}, \bm{\phi}_{f^d})_{T^d}.
		\end{align*}
		Therefore, we can immediately observe that it is enough to show that $(K^{-1}(\bm{v}_{0h}+ R_h \mu_h), \bm{v}_{0h} + R_h \mu_h)_{T^d}$ and $(K^{-1}(\bm{v}_{0h}+ R_h \mu_h), \bm{v}_{0h} + R_h \mu_h)_{D,T^d}$ are spectrally equivalent on each element $T^d$, $0< d \leq n$.  In addition, by using the scaling argument~\cite[Section 4.5.2]{brenner}, we only need to show they are spectrally equivalent on a reference element $\hat{T}^d$, i.e., 
		\begin{align}
			\quad \tilde{c}_1  (K^{-1}(\bm{v}_{0h} & + R_h \mu_h), \bm{v}_{0h} + R_h \mu_h)_{D,\hat{T}^d} \nonumber \\
			&\leq (K^{-1}(\bm{v}_{0h}+ R_h \mu_h), \bm{v}_{0h} + R_h \mu_h)_{\hat{T}^d} \label{sec:4.eq:norm_equiv_reference} \\ 
			&\leq \tilde{c}_2 (K^{-1}(\bm{v}_{0h}+ R_h \mu_h), \bm{v}_{0h} + R_h \mu_h)_{D,\hat{T}^d}. \nonumber
		\end{align} 
		We show the proof for $d=n=3$. For other cases the proof follows similarly.  

		For $ d = n = 3$, the reference element $ \hat{T}^d $ is a tetrahedron with vertices $ (0, 0, 0) $, $ (1, 0, 0) $, $ (0, 1, 0) $ and $ (0, 0, 1) $ in the Cartesian coordinates.  The local matrix $ A_{\hat{T}^d} $, representing $(K^{-1}(\bm{v}_{0h}+ R_h \mu_h), \bm{v}_{0h} + R_h \mu_h)_{\hat{T}^d}$,  takes the following form
		\begin{equation*}
			A_{\hat{T}^d} = \dfrac{K^{-1}}{120}
			\begin{pmatrix}
				18 & \sqrt{3} & \sqrt{3} & \sqrt{3} \\
				\sqrt{3} & 16 & -4 & -4 \\
				\sqrt{3} & -4 & 16 & -4 \\
				\sqrt{3} & -4 & -4 & 16
			\end{pmatrix},
		\end{equation*} 

		By the definition, $(K^{-1}(\bm{v}_{0h}+ R_h \mu_h), \bm{v}_{0h} + R_h \mu_h)_{D, \hat{T}^d}$ is represented by the diagonal of $A_{\hat{T}^d}$, which we denote as $D_{A_{\hat{T}^d}} = \frac{K^{-1}}{120} \text{diag}(18,16,16,16)$. To show~\eqref{sec:4.eq:norm_equiv_reference} on  $ \hat{T}^d $, it is enough to notice that, under our assumption that $K$ is constant on each $T^d$, the generalized eigenvalue problem $ A_{\hat{T}^d} \bm{y} = \chi D_{A_{\hat{T}^d}} \bm{y} $ gives all eigenvalues $ \chi > 0 $ independent of physical and discretization parameters. Therefore, \eqref{sec:4.eq:norm_equiv_reference} holds with $ \tilde{c}_1 = \sqrt{\chi_{\min}} $ and $ \tilde{c}_2 = \sqrt{\chi_{\max}}$, where $\chi_{\min}$ and $\chi_{\max}$ denote the smallest and largest eigenvalue, respectively.  The spectral equivalent result~\eqref{sec:4.eq:norm_equi} follows directly by the scaling argument~\cite[Section 4.5.2]{brenner} and summing over all $T^d \in \mathcal{T}^d_{\Omega}$, $0\leq d \leq n$.  The constants $c_1$ and $c_2$  depend on the shape regularity of the mesh due to the scaling argument but do not depend on the physical and discretization parameters.
	\end{proof}

	Based on the spectral equivalence Lemma~\eqref{sec:4.lemma:norm_equi}, we have the following inf-sup condition regarding the discrete gradient $\bm{D}_h$.

	\begin{lemma}
		\label{sec:4.lemma:infsup}
		Let the discrete gradient operator $ \bm{D}_h $ be defined as in \eqref{sec:4.eq:D_h}. Then there exists a constant $ \beta_{\star} > 0 $ independent of the discretization and physical parameters such that
		\begin{equation}
			\label{sec:4.eq:infsup}
			\inf\limits_{q_h \in Q_h} \sup\limits_{\bm{r}_h \in \bm{W}_h}
			\dfrac{ a^D(\bm{D}_h q_h, \bm{r}_h)}{\| \bm{r}_h \|_{a} \| \bm{D}_h q_h \|_{a^D}} \geq \beta_{\star}.
		\end{equation}
	\end{lemma}

	\begin{proof}
		Using Lemma \ref{sec:4.lemma:norm_equi}, we have for any $ q_h \in Q_h $
		\begin{align*}
			\sup\limits_{\bm{r}_h \in \bm{W}_h}
			\dfrac{ a^D(\bm{D}_h q_h, \bm{r}_h)}{\| \bm{r}_h \|_{a}} 
			& \geq \sup\limits_{\bm{r}_h \in \bm{W}_h}
			\dfrac{ a^D(\bm{D}_h q_h, \bm{r}_h)}{ c_2 \| \bm{r}_h \|_{a^D}} \\
			& = c_2^{-1} \| \bm{D}_h q_h \|_{a^D}.
		\end{align*}
		Now the result follows taking infimum over all $ q_h \in Q_h $ and $ \beta_{\star} = c_2^{-1} $.
	\end{proof}

	Based on Lemma~\ref{sec:4.lemma:norm_equi} and~\ref{sec:4.lemma:infsup}, by Babuska-Brezzi theory \cite{BoffiBrezziFortin2013, Brezzi1974}, we can conclude that the alternative formulation~\eqref{sec:4.eq:alter-saddle_point_h} is well-posed with respect to the norm~\eqref{sec:4.eq:alter_norm} as stated in the following theorem. 
	
	\begin{theorem}\label{sec:4.thm:alter_form_well_posed}
		Consider the composite bilinear form on the space $\bm{W}_h \times Q_h$,
		\begin{equation*}
			\mathcal{L}(\bm{w}_h, p_h; \bm{r}_h, q_h) := a(\bm{w}_h,\bm{r}_h) + a^D(\bm{D}_h p_h, \bm{r}_h) + a^D(\bm{D}_hq_h, \bm{w}_h).
		\end{equation*}
		It satisfies the continuity condition and the inf-sup condition with respect to $\vertiii{(\bm{r}_h, q_h)}$, i.e., for any $(\bm{w}_h,p_h) \in \bm{W}_h \times Q_h$ and $(\bm{r}_h, q_h) \in \bm{W}_h \times Q_h$, 
		\begin{align*}
			|\mathcal{L}(\bm{w}_h, p_h; \bm{r}_h, q_h)| &\leq \alpha \vertiii{(\bm{w}_h, p_h)}\vertiii{(\bm{r}_h, q_h)}, \\
			\inf_{(\bm{r}_h, q_h) \in \bm{W}_h \times Q_h}\sup_{(\bm{w}_h,p_h) \in \bm{W}_h \times Q_h} \frac{\mathcal{L}(\bm{w}_h, p_h; \bm{r}_h, q_h)}{\vertiii{(\bm{w}_h, p_h)}\vertiii{(\bm{r}_h, q_h)}} &\geq \beta, 
		\end{align*}
		with constants $\alpha$ and $\beta$ dependent on the shape regularity of the mesh but independent of discretization and physical parameters.
	\end{theorem}

	\subsection{Block diagonal preconditioners}
	\label{sec:4.subsec:diag_precond}

	The well-posedness Theorem~\ref{sec:4.thm:alter_form_well_posed} provides alternative block preconditioners for solving the linear system~\eqref{sec:2.eq:block_form_h} effectively.  To this end, we introduce a linear operators $D_A: \bm{W}_h \to \bm{W}_h'$ which is defined as $\langle D_A \bm{w}_h, \bm{r}_h \rangle = a^D(\bm{w}_h, \bm{r}_h)$ for $ \bm{w}_h, \bm{r}_h \in \bm{W}_h $.  The reason we use the notation $D_A$ here is that, by the definitions of $a(\cdot,\cdot)$ and $a^D(\cdot, \cdot)$, the matrix representation of linear operator $D_A$ is exactly the diagonal of the matrix representation of linear operator $A$.  Then, by the definition of the discrete gradient operator $\bm{D}_h$~\eqref{sec:4.eq:D_h}, we have $D_A \bm{D}_h = B^T$ and, therefore, 
	\begin{equation*}
		\| \bm{D}_h q_h \|^2_{a^D} = \langle D_A \bm{D}_h q_h, \bm{D}_h q_h \rangle = \langle B^T q_h, D_A^{-1} B^T q_h \rangle = \langle B D_A^{-1} B^T q_h, q_h \rangle,
	\end{equation*}
	for $ q_h \in Q_h $. Based on the above operator form of the $\| \cdot \|_{a^D}$ norm, the Riesz operator corresponding to the $\vertiii{\cdot}$ norm~\eqref{sec:4.eq:alter_norm} is

	\begin{equation}
		\label{sec:4.eq:diag_precond}
		\mathcal{B}_D = 
		\begin{pmatrix}
			A & 0 \\
			0 & B D_A^{-1} B^T
		\end{pmatrix}^{-1}.
	\end{equation}

	As discussed in Section~\ref{sec:3.subsec:norm-equiv}, $\mathcal{B}_D$ is a norm-equivalent preconditioner for solving the system~\eqref{sec:2.eq:block_form_h} and we have the following theorem regarding the condition number of $\mathcal{B}_D \mathcal{A}$. 

	\begin{theorem}
		\label{sec:4.thm:cond_number}
		Let $ \mathcal{B}_D $ be as in \eqref{sec:4.eq:diag_precond}. Then $ \kappa(\mathcal{B}_D \mathcal{A}) \leq \dfrac{\alpha}{\beta}$. 
	\end{theorem}


	\begin{remark}
		\label{sec:4.rem:robustness}
		Notice that Theorem~\ref{sec:4.thm:alter_form_well_posed} (essentially Lemma~\ref{sec:4.lemma:norm_equi}) ensures that $ \kappa(\mathcal{B}_D \mathcal{A}) $ is bounded independently of $ h $ and parameters $ K $, $K_{\bm{\nu}}$ and $ \gamma $, but remains dependent on the shape regularity of the mesh.
	\end{remark}



	In practice, applying the preconditioner $ \mathcal{B}_D $ implies inverting the diagonal block exactly, which can be expensive and sometimes infeasible. Thus, we consider the following preconditioner 
	\begin{equation}
		\label{sec:4.eq:spectral_equi_precond}
		\mathcal{M}_D = 
		\begin{pmatrix}
			H_{\bm{w}} & 0 \\
			0 & H_p
		\end{pmatrix},
	\end{equation}
	where the diagonal blocks $ H_{\bm{w}} $ and $ H_{p} $ are symmetric positive definite and spectrally equivalent to diagonal blocks in $A$ and $BD_A^{-1}B^T$, respectively, i.e.
	\begin{subequations}
		\label{sec:4.eq:spectral_equi_blocks}
		\begin{align}
			c_{1, \bm{w}} (H_{\bm{w}} \bm{r}_h, \bm{r}_h) & \leq (A^{-1} \bm{r}_h , \bm{r}_h ) \leq c_{2, \bm{w}} (H_{\bm{w}} \bm{r}_h, \bm{r}_h), \label{sec:4.eq:spectral_equi_blocks_a}\\
			c_{1, p} (H_p q_h, q_h) & \leq ( (B D_A^{-1} B^T)^{-1} q_h , q_h) \leq c_{2, p} (H_p q_h, q_h), \label{sec:4.eq:spectral_equi_blocks_b}
		\end{align}
	\end{subequations}	
	where the constants $c_{1,\bm{w}}$, $c_{1,p}$, $c_{2,\bm{w}}$, and $c_{2,p}$ are independent of discretization and physical parameters.  In practice, $H_{\bm{w}}$ can be defined by a diagonal scaling, i.e., $D_A^{-1}$ and $H_p$ can be defined by standard multigrid methods.  In general, the choice of $H_{\bm{w}}$ and $H_p$ are not very restrictive, provided it handles possible heterogeneity in physical parameters $K$, $K_{\nu}$, and $\gamma$. 

	$\mathcal{M}_D$ is a norm-equivalent preconditioner as well. Following~\cite{mardal:winther}, we can directly estimate the condition number of $\mathcal{M}_D \mathcal{A}$ in the following theorem.
	\begin{theorem}
		\label{sec:4.thm:cond_number_MD}
		Let $\mathcal{M_D}$ be as in~\eqref{sec:4.eq:spectral_equi_precond} and let \eqref{sec:4.eq:spectral_equi_blocks} hold. Then it follows that $\kappa(\mathcal{M}_D\mathcal{A}) \leq \frac{\alpha c_2}{ \beta c_1}$, where $c_2 = \max\{c_{2,\bm{w}}, c_{2,p}\}$ and $c_1 = \min \{ c_{1,\bm{w}}, c_{1,p} \}$.	
	\end{theorem}

	\begin{remark}
		\label{sec:4.rem:robustness_MD}
		Again, $ \kappa(\mathcal{M}_D \mathcal{A}) $ is bounded independently of $ h $ and parameters $ K $, $K_{\bm{\nu}}$ and $ \gamma $, but remains dependent on the shape regularity of the mesh.
	\end{remark}




\subsection{Block triangular preconditioners}
\label{sec:4.subsec:triang_precond}
	In this subsection, we consider the block triangular preconditioners based on the FOV-equivalent preconditioners we discussed in Section~\ref{sec:3.subsec:FOV-equiv}.   Here, we analyze the robustness of block triangular preconditioners and show the corresponding FOV-equivalence, which leads to uniform convergence rate of the GMRES method. 

	The block lower triangular preconditioners take the following form
	\begin{equation}
		\label{sec:3.eq:lower_precond}
		\mathcal{B}_L = 
		\begin{pmatrix}
			A & 0 \\
			-B & B D_A^{-1} B^T
		\end{pmatrix}^{-1} 
		\text{ and }
		\mathcal{M}_L = 
		\begin{pmatrix}
			H_{\bm{w}}^{-1} & 0 \\
			-B & H_p^{-1}
		\end{pmatrix}^{-1}.
	\end{equation}
	On the other hand, the block upper triangular preconditioners are given as
	\begin{equation}
		\label{sec:3.eq:upper_precond}
		\mathcal{B}_U = 
		\begin{pmatrix}
			A & B^T \\
			0 & B D_A^{-1} B^T
		\end{pmatrix}^{-1}
		\text{ and }
		\mathcal{M}_U = 
		\begin{pmatrix}
			H_{\bm{w}}^{-1} & B^T \\
			0 & H_p^{-1}
		\end{pmatrix}^{-1}.
	\end{equation}
	Basically, $\mathcal{M}_L$ and $\mathcal{M}_U$ are inexact versions of $\mathcal{B}_L$ and $\mathcal{B}_U$ when the diagonal blocks are replaced by spectrally equivalent approximations~\eqref{sec:4.eq:spectral_equi_blocks}.


	Next theorem shows that $\mathcal{B}_L$ and $\mathcal{A}$ are FOV-equivalent.
	\begin{theorem}
		\label{sec:4.thm:FOV_lower}
		There exist constants $ \xi_1, \xi_2 > 0 $, independent of discretization and physical parameters, such that for every $ \bm{x} = (\bm{w}_h, p_h) \in \bm{W}_h \times Q_h $, $ \bm{x} \neq \bm{0} $,
		\begin{align*}
			\xi_1 \leq \dfrac{ (\mathcal{B}_L \mathcal{A} \bm{x}, \bm{x})_{\mathcal{B}_D^{-1}} }{ (\bm{x}, \bm{x})_{\mathcal{B}_D^{-1}} }, & \quad \text{ and } \quad 
			\dfrac{ \| \mathcal{B}_L \mathcal{A} \bm{x} \|_{\mathcal{B}_D^{-1}} }{ \| \bm{x} \|_{\mathcal{B}_D^{-1}} } \leq \xi_2.
		\end{align*}
	\end{theorem}

	\begin{proof}
		By the definition of the linear operators $A$ and $D_A$, we naturally have $\| \bm{w}_h \|_a = \| \bm{w}_h \|_A$ and $\| \bm{w}_h \|_{a^D} = \| \bm{w}_h \|_{D_A}$, respectively. Here $\| \bm{w}_h \|_A^2 := \langle A \bm{w}_h, \bm{w}_h \rangle$ and $\| \bm{w}_h \|_{D_A}^2 := \langle D_A \bm{w}_h, \bm{w}_h \rangle$ for $\bm{w}_h \in \bm{W}_h$.
		
		Then Lemma~\ref{sec:4.lemma:norm_equi} states that the norms $  \| \cdot \|_{D_A}  $ and $ \| \cdot \|_{A} $ are equivalent, which also implies the equivalence between the norms $  \| \cdot \|_{D_A^{-1}}  $ and $ \| \cdot \|_{A^{-1}} $, which are defined as $\| \bm{w}'_h \|_{A^{-1}}^2 := \langle A^{-1} \bm{w}'_h, \bm{w}'_h \rangle$ and $\| \bm{w}'_h \|_{D_A^{-1}}^2 := \langle D_A^{-1} \bm{w}'_h, \bm{w}'_h \rangle$ for $\bm{w}'_h \in \bm{W}'_h$.
		
		Using that and Cauchy-Schwarz inequality, we have
			\begin{align*}
				(\mathcal{B}_L \mathcal{A} \bm{x}, \bm{x})_{\mathcal{B}_D^{-1}} 
				& = \| \bm{w}_h \|^2_A + \langle B^T p_h, \bm{w}_h \rangle + \| B A^{-1} B^T p_h \|^2 \\
				& \geq  \| \bm{w}_h \|^2_A - \| B^T p_h \|_{A^{-1}} \| \bm{w}_h \|_A + \| B^T p_h \|_{A^{-1}}^2 \\
				& = 
				\begin{pmatrix} \| \bm{w}_h \|_A \\  \| B^T p_h \|_{A^{-1}} \end{pmatrix}^T
				\begin{pmatrix}
					1 & -\frac{1}{2} \\
					-\frac{1}{2} & 1
				\end{pmatrix}
				\begin{pmatrix} \| \bm{w}_h \|_A \\  \| B^T p_h \|_{A^{-1}} \end{pmatrix} \\
				& \geq \dfrac{1}{2} ( \| \bm{w}_h \|^2_A + \| B^T p_h \|_{A^{-1}}^2  ) \\
				& \geq \dfrac{1}{2} ( \| \bm{w}_h \|^2_A + c_2^{-1} \| B^T p_h \|_{D_A^{-1}}^2  ) \\
				& \geq \xi_1 \| \bm{x} \|^2_{\mathcal{B}_D^{-1}},
			\end{align*}
			with $ \xi_1 =  \dfrac{1}{2} \min \left\{ 1,  c_2^{-1} \right\} $. On the other hand, again using the Cauchy-Schwarz inequality and equivalence of the norms $  \| \cdot \|_{D_A^{-1}}  $ and $ \| \cdot \|_{A^{-1}} $ we get
			\begin{align*}
				(\mathcal{B}_L \mathcal{A} \bm{x}, \bm{y})_{\mathcal{B}_D^{-1}} 
				& = \langle A \bm{w}_h, \bm{r}_h \rangle + \langle B^T p_h, \bm{r}_h \rangle + \langle B A^{-1} B^T p_h, q_h \rangle \\
				& \leq \| \bm{w}_h \|_A \| \bm{r}_h \|_A + \| B^T p_h \|_{A^{-1}} \| \bm{r}_h \|_A + \| B^T p_h \|_{A^{-1}} \| B^T q_h \|_{A^{-1}} \\
				& \leq \left(  \| \bm{w}_h \|^2_A + 2 \| B^T p_h \|^2_{A^{-1}} \right)^{\frac{1}{2}}
				\left( 2 \| \bm{r}_h \|^2_A + \| B^T q_h \|^2_{A^{-1}} \right)^{\frac{1}{2}} \\
				&  \leq \left(  \| \bm{w}_h \|^2_A + 2 c_1^{-1} \| B^T p_h \|_{D_A^{-1}}^2 \right)^{\frac{1}{2}}
				\left(  2 \| \bm{r}_h \|^2_A + c_1^{-1} \| B^T q_h \|_{D_A^{-1}}^2 \right)^{\frac{1}{2}} \\ 
				& \leq \xi_2 \| \bm{x} \|_{\mathcal{B}_D^{-1}} \| \bm{y} \|_{\mathcal{B}_D^{-1}},
			\end{align*}
			for each $ \bm{y} = (\bm{r}_h, q_h) \in \bm{W}_h \times Q_h, \bm{y} \not = \bm{0} $ with $ \xi_2 = \max \left\{ 2, 2 c_1^{-1} \right\} $, which concludes the proof.
	\end{proof}

	The next theorem states that if the conditions~\eqref{sec:4.eq:spectral_equi_blocks} hold then $\mathcal{M}_L$ and $\mathcal{A}$ are FOV-equivalent. 

	\begin{theorem}
		\label{sec:4.thm:FOV_lower_approx}
		If the conditions \eqref{sec:4.eq:spectral_equi_blocks} hold and $ \| I - H_{\bm{w}} A \|_A \leq \rho $ for $ 0 \leq \rho < 1 $, then there exist constants $ \xi_1, \xi_2 > 0 $ independent of discretization and physical parameters such that for every $ \bm{x} = (\bm{w}_h, p_h) \in \bm{W}_h \times Q_h $, $ \bm{x} \neq \bm{0} $,
		\begin{align*}
			\xi_1 \leq \dfrac{ (\mathcal{M}_L \mathcal{A} \bm{x}, \bm{x})_{\mathcal{M}_D^{-1}} }{ (\bm{x}, \bm{x})_{\mathcal{M}_D^{-1}} }, & \quad \text{ and } \quad
			\dfrac{ \| \mathcal{M}_L \mathcal{A} \bm{x} \|_{\mathcal{M}_D^{-1}} }{ \| \bm{x} \|_{\mathcal{M}_D^{-1}} } \leq \xi_2.
		\end{align*}
	\end{theorem}

	\begin{proof}
		From the assumptions of the theorem we have $ \| H_{\bm{w}} A \|_A \leq 1 + \rho $ in combination with Lemma \ref{sec:4.lemma:norm_equi}, \eqref{sec:4.eq:spectral_equi_blocks} and the Cauchy-Schwarz inequality, we have that
		\begin{align*}
			(\mathcal{M}_L \mathcal{A} \bm{x}, \bm{x})_{\mathcal{M}_D^{-1}} 
			& = \| \bm{w}_h \|^2_A + \langle B^T p_h, \bm{w}_h \rangle + \langle B(H_{\bm{w}}A - I) \bm{w}_h, p_h \rangle + \| B^T p_h \|_{H_{\bm{w}}}^2 \\
			& = \| \bm{w}_h \|^2_A + \langle H_{\bm{w}} A \bm{w}_h, B^T p_h \rangle + \| B^T p_h \|_{H_{\bm{w}}}^2 \\
			& \geq  \| \bm{w}_h \|^2_A - (1 + \rho) \| \bm{w}_h \|_A \| B^T p_h \|_{H_{\bm{w}}}  + \| B^T p_h \|_{H_{\bm{w}}}^2 \\
			& = 
			\begin{pmatrix} \| \bm{w}_h \|_A \\  \| B^T p_h \|_{H_{\bm{w}}} \end{pmatrix}^T
			\begin{pmatrix}
				1 & -\frac{ 1 + \rho }{2} \\
				-\frac{1 + \rho}{2} & 1
			\end{pmatrix}
			\begin{pmatrix} \| \bm{w}_h \|_A \\  \| B^T p_h \|_{H_{\bm{w}}} \end{pmatrix} \\
			& \geq \frac{ 1 - \rho }{2} ( \| \bm{w}_h \|^2_A + \| B^T p_h \|_{H_{\bm{w}}}^2  ) \\
			& \geq \frac{ 1 - \rho }{2} \left( c_{2, \bm{w}}^{-1} \| \bm{w}_h \|^2_{H_{\bm{w}}^{-1}} +  c_{1, \bm{w}} c_2^{-1} \| B^T p_h \|_{D_A^{-1}}^2  \right) \\
			& \geq \frac{ 1 - \rho }{2} \left( c_{2, \bm{w}}^{-1} \| \bm{w}_h \|^2_{H_{\bm{w}}^{-1}} + c_{1, \bm{w}} c_2^{-1} c_{2, p}^{-1} \| p_h \|_{H_p^{-1}}^2  \right) \\
			& \geq \xi_1 \| \bm{x} \|^2_{\mathcal{M}_D^{-1}},
		\end{align*}
		with $ \xi_1 = \frac{ 1 - \rho }{2} \min \left\{ c_{2, \bm{w}}^{-1}, c_{1, \bm{w}} c_2^{-1} c_{2, p}^{-1} \right\}  $. 
		
		Using the same conditions to show the upper bound, we obtain
		\begin{align*}
			(\mathcal{M}_L \mathcal{A} \bm{x}, \bm{y})_{\mathcal{M}_D^{-1}} 
			& = \langle A \bm{w}_h, \bm{r}_h \rangle + \langle B^T p_h, \bm{r}_h \rangle + \langle B(H_{\bm{w}}A - I) \bm{w}_h, q_h \rangle + \langle B H_{\bm{w}} B^T p_h, q_h \rangle\\
			& \leq \| \bm{w}_h \|_A \| \bm{r}_h \|_A + \| B^T p_h \|_{A^{-1}} \| \bm{r}_h \|_{A} + \| (H_{\bm{w}}A - I) \bm{w}_h \|_{A} \| B^T q_h \|_{A^{-1}} \\ 
			& \quad + \| B^T p_h \|_{H_{\bm{w}}} \| B^T q_h \|_{H_{\bm{w}}} \\
			& \leq \| \bm{w}_h \|_A \| \bm{r}_h \|_A + \| B^T p_h \|_{A^{-1}} \| \bm{r}_h \|_{A} + \rho \| \bm{w}_h \|_{A} \| B^T q_h \|_{A^{-1}} \\ 
			& \quad + \| B^T p_h \|_{H_{\bm{w}}} \| B^T q_h \|_{H_{\bm{w}}} \\
			& \leq \left( (1 + \rho^2 ) \| \bm{w}_h \|^2_A + \| B^T p_h \|^2_{A^{-1}} + \| B^T p_h \|^2_{H_{\bm{w}}} \right)^{\frac{1}{2}} \\
			& \quad \left( 2 \| \bm{r}_h \|^2_A + \| B^T q_h \|^2_{A^{-1}} + \| B^T q_h \|^2_{H_{\bm{w}}}  \right)^{\frac{1}{2}} \\
			& \leq \left( (1 + \rho^2 ) c^{-1}_{1, \bm{w}} \| \bm{w}_h \|^2_{H_{\bm{w}}^{-1}} + c^{-1}_1(1 + c^{-1}_{1, \bm{w}})  \| B^T p_h \|^2_{D_A^{-1}} \right)^{\frac{1}{2}} \\
			& \quad \left( 2 c^{-1}_{1, \bm{w}} \| \bm{r}_h \|^2_{H_{\bm{w}}^{-1}} + c^{-1}_1(1 + c^{-1}_{1, \bm{w}})  \| B^T q_h \|^2_{D_A^{-1}} \right)^{\frac{1}{2}} \\
			& \leq \left( (1 + \rho^2 ) c^{-1}_{1, \bm{w}} \| \bm{w}_h \|^2_{H_{\bm{w}}^{-1}} + c_{1,p}^{-1}c^{-1}_1(1 + c^{-1}_{1, \bm{w}})  \| p_h \|^2_{H_p^{-1}} \right)^{\frac{1}{2}}\\
			& \quad \left( 2 c^{-1}_{1, \bm{w}} \| \bm{r}_h \|^2_{H_{\bm{w}}^{-1}} + c_{1,p}^{-1}c^{-1}_1(1 + c^{-1}_{1, \bm{w}})  \| q_h \|^2_{H_p^{-1}} \right)^{\frac{1}{2}}  \\
			& \leq \xi_2 \| \bm{x} \|_{\mathcal{M}_D^{-1}} \| \bm{y} \|_{\mathcal{M}_D^{-1}}.
		\end{align*}
		This gives the upper bound with $ \xi_2 = \max \{ 2c_{1, \bm{w}}^{-1}, c_{1,p}^{-1}c_1^{-1}(1+c_{1,\bm{w}}^{-1}) \} $, which concludes the proof. 
	\end{proof}

	\begin{remark}
		\label{sec:4.rem:robustness_lower}
		Due to Lemma~\ref{sec:4.lemma:norm_equi}, the constants $\xi_1$ and $\xi_2$ are independent of $ h $ and parameters $ K $, $K_{\bm{\nu}}$ and $ \gamma $, but remain dependent on the shape regularity of the mesh. This means that the convergence rate of the preconditioned GMRES method with preconditioner $\mathcal{B}_L$ or $\mathcal{M}_L$ depends only on the shape regularity of the mesh.
	\end{remark}

%

	Similarly, we can derive the FOV-equivalence of $ \mathcal{B}_U $ and $ \mathcal{M}_U $ with $ \mathcal{A} $. Since the proofs are similar to the two previous theorems, we omit them and only state the results here. 

	\begin{theorem}
		\label{sec:5.thm:FOV_upper}
		There exist constants $ \xi_1, \xi_2 > 0 $ independent of discretization and physical parameters such that for any $\bm{x}' = \mathcal{B}_U^{-1} \bm{x}$ with $\bm{x} = (\bm{w}_h, p_h) \in \bm{W}_h \times Q_h $, $\bm{x} \neq \bm{0}$,
		\begin{align*}
			\xi_1 \leq \dfrac{ ( \mathcal{A} \mathcal{B}_U \bm{x}', \bm{x}')_{\mathcal{B}_D} }{ (\bm{x}', \bm{x}')_{\mathcal{B}_D} }, & \quad \text{ and } \quad 
			\dfrac{ \| \mathcal{A} \mathcal{B}_U \bm{x}' \|_{\mathcal{B}_D} }{ \| \bm{x}' \|_{\mathcal{B}_D} } \leq \xi_2.
		\end{align*}
	\end{theorem}

	\begin{theorem}
		\label{sec:4.thm:FOV_upper_approx}
		If the conditions \eqref{sec:4.eq:spectral_equi_blocks} hold and $ \| I - H_{\bm{w}} A \|_A \leq \rho $ for $ 0 \leq \rho < 1 $, then there exist constants $ \xi_1, \xi_2 > 0 $ independent of discretization and physical parameters such that for any $\bm{x}' = \mathcal{M}_U^{-1} \bm{x}$ with $\bm{x} = (\bm{w}_h, p_h) \in \bm{W}_h \times Q_h $, $\bm{x} \neq \bm{0}$,
		\begin{align*}
			\xi_1 \leq \dfrac{ (\mathcal{A} \mathcal{M}_U  \bm{x}', \bm{x}')_{\mathcal{M}_D} }{ (\bm{x}', \bm{x}')_{\mathcal{M}_D} }, & \quad \text{ and } \quad
			\dfrac{ \| \mathcal{A} \mathcal{M}_U \bm{x}' \|_{\mathcal{M}_D} }{ \| \bm{x}' \|_{\mathcal{M}_D} } \leq \xi_2.
		\end{align*}
	\end{theorem}

	\begin{remark}
		\label{sec:4.rem:robustness_lower_inexact}
		Similarly, the constants $\xi_1$ and $\xi_2$ here are independent of $ h $ and parameters $ K $, $K_{\bm{\nu}}$ and $ \gamma $, but remain dependent on the shape regularity of the mesh. This means that the convergence rate of the preconditioned GMRES method with preconditioner $\mathcal{B}_U$ or $\mathcal{M}_U$ depends only on the shape regularity of the mesh.
	\end{remark}

	
	\section{Numerical results}
	\label{sec:results}
	In this section, we propose several test cases to verify the theory on the robustness of the preconditioners derived above. Both two and three dimensional examples emphasize common challenges in fracture flow simulations such as large aspect ratios of rock and fractures, complex fracture network structures and high heterogeneity in the permeability fields.
	
	In each example below, a set of mixed-dimensional simplicial grids is generated on rock and fracture subdomains, where the coupling between the rock and fracture is employed by a separate mortar grid. Since our main objective is to show the robustness of our preconditioners for standard Krylov iterative methods, for the sake of simplicity, we take the mortar grid to be matching with the adjacent subdomain grids. However, the theory in Section \ref{sec:precond_MD} shows no restrictions to relative grid resolution between the rock, fracture and mortar grids. Furthermore, in \cite{Nordbotten2018} the discrete system remains well-posed with varying coarsening/refinement ratio for non-degenerate (normal) permeability values, which is one of our assumptions. Therefore, we expect that our block preconditioners give similar performance for general grids between the rock, fracture, and coupling part. 
		
	To solve the system \eqref{sec:2.eq:block_form_h}, we use a Flexible Generalized Minimal Residual (FGMRES) method as an \textit{outer} iterative solver, with the tolerance for the relative residual set to $ 10^{-6} $. The block preconditioners designed in Section \ref{sec:precond_MD} are used to accelerate the convergence rate of FGMRES. Each preconditioner $ \mathcal{B}_D, \mathcal{B}_L $ and $ \mathcal{B}_U $ requires inversion of the diagonal blocks corresponding to flux and pressure degrees of freedom, while the spectrally equivalent versions $ \mathcal{M}_D, \mathcal{M}_L $ and $ \mathcal{M}_U $ approximate the inverses with appropriate iterative methods. For that, we implement both exact and inexact \textit{inner} solvers. Solving each diagonal blocks exactly means we use the GMRES method with a relative residual tolerance set to $ 10^{-10}$, while in the inexact case it is set to $ 10^{-3} $. Inner GMRES is preconditioned with unsmoothed aggregation Algebraic Multigrid method (AMG) in a W-cycle.
	
	For obtaining the mixed-dimensional geometry and discretization, we use the PorePy library \cite{Keilegavlen2017a}, an open-source simulation tool for fractured and deformable porous media written in Python. Our preconditioners are implemented in HAZMATH library \cite{hazmath}, a finite element solver library written in C, also where all solving computations are performed. The numerical tests were performed on a workstation with an $8$-core 3GHz Intel Xeon ``Sandy Bridge" CPU and 256 GB of RAM.

	
	\subsection{Example: two-dimensional Geiger network}
	\label{subsec:example1}
	
	In the first example, we consider the test case presented in the benchmark study \cite{Flemisch2016a}. The domain $ \Omega = (0, 1)^2 $, depicted in Figure \ref{fig:geiger_domain}, has unitary permeability $ K = \bm{I} $ for the rock matrix and it is divided into 10 sub-domains by a set of fractures with aperture $ \gamma $. In our case, we set the tangential and normal permeability of the fractures to be constant throughout the whole network, and vary the value from blocking to conducting the flow. The tangential fracture permeability is denoted as $ K_f $ to avoid confusion with the rock permeability. At the boundary, we impose zero flux condition on the top and bottom, unitary pressure on the right, and flux equal to $-1$ on the left. The boundary conditions are applied to both the rock matrix and the fracture network. The numerical solution to this problem is also illustrated in Figure \ref{fig:geiger_domain}.
	
	\begin{figure}[htbp]
		\centering
		\includegraphics[width=0.4\textwidth]{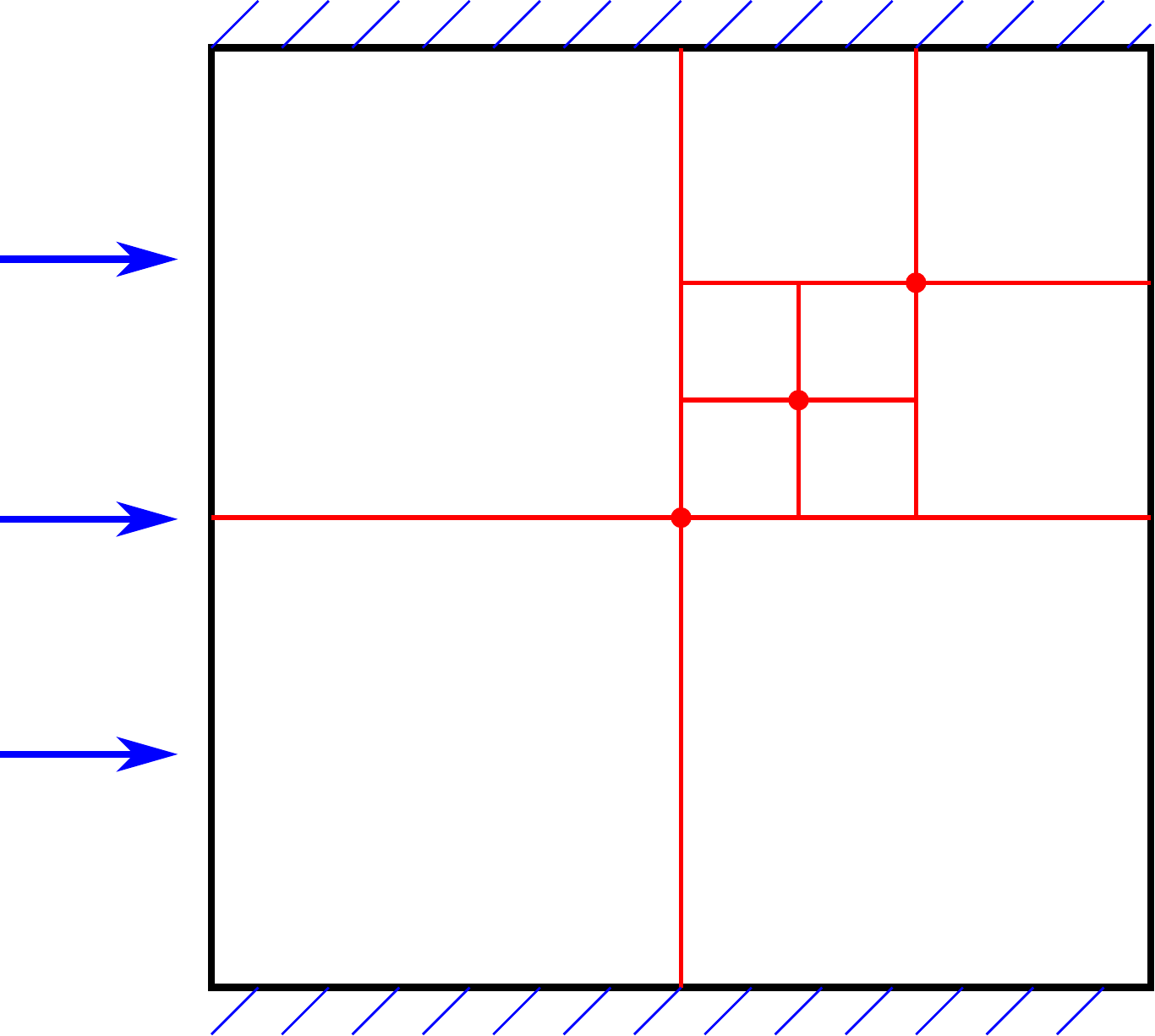}
		\hspace{0.5cm}
		\includegraphics[width=0.45\textwidth]{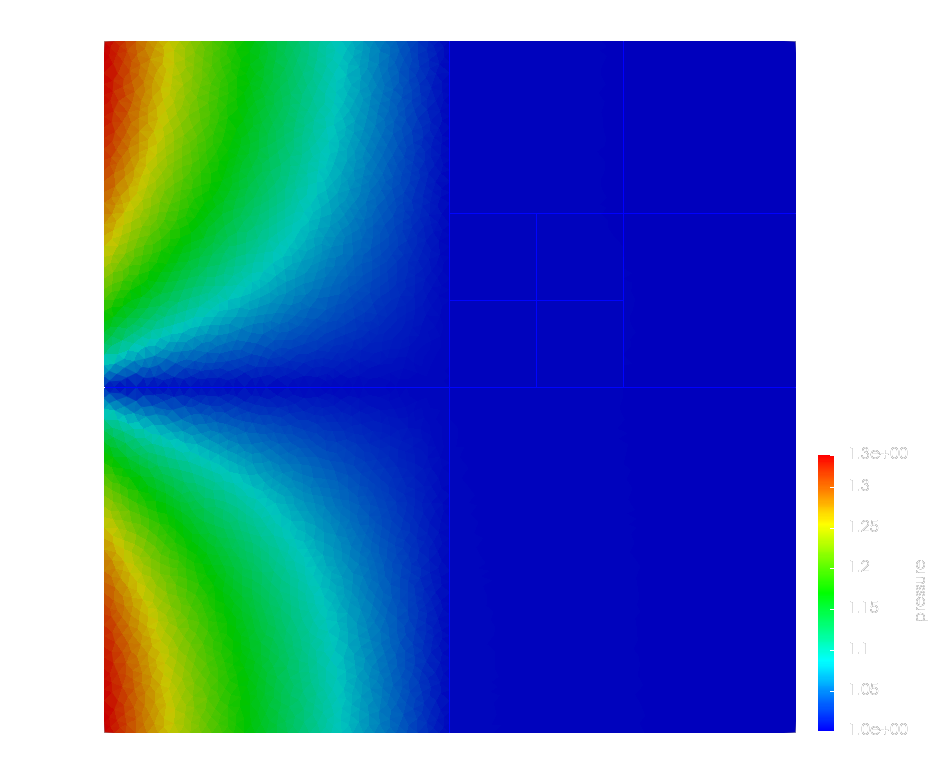}
		\caption{(Left) Graphical representation of the domain and fracture network
			geometry of Example \ref{subsec:example1}. (Right) Pressure solution for a case of conducting fractures.}
		\label{fig:geiger_domain}
	\end{figure}
	
	Our goal is to investigate the robustness of the block preconditioners with respect to discretization parameter $ h $ and physical parameters $ \gamma $, $ K_f $ and $ K_{\bm{\nu}} $. To this end, we generate a series of tests in which we vary the magnitude of one of the parameters, while setting others to a fixed value. This also  tests the heterogeneity ratios between the porous medium and the fractures, since we keep spatial and physical parameters of the porous medium unitary. We compute and compare number of iterations of the outer solver for both exact and inexact implementations of the proposed preconditioners. This way we clearly see if the stability of the proposed preconditioners depends on one or a combination of given parameters. 
	
	\begin{table}[htbp]
		\centering
		\begin{tabular}{r|r|r|r||r|r|r|}
			\cline{2-7}
			& \multicolumn{3}{c||}{\textit{Inexact}} & \multicolumn{3}{c|}{\textit{Exact}} \\
			\hline \hline
			\multicolumn{1}{|c||}{$h$} & $\mathcal{M}_D$ & $\mathcal{M}_L$ & $\mathcal{M}_U$ & $\mathcal{B}_D$ & $\mathcal{B}_L$ & $\mathcal{B}_U$ \\ \hline
			\multicolumn{1}{|l||}{$1/4$} & 20 & 13 & 12 & 19 & 10 & 10  \\ \hline
			\multicolumn{1}{|l||}{$1/8$} & 19 & 13 & 11 & 19 & 10 & 10    \\ \hline
			\multicolumn{1}{|l||}{$1/16$} & 19 & 13 & 11 & 19 & 10 & 10   \\ \hline
			\multicolumn{1}{|l||}{$1/32$} & 19 & 13 & 11 & 19 & 10 & 10    \\ \hline
			\multicolumn{1}{|l||}{$1/64$} & 19 & 13 & 11 & 19 & 10 & 10    \\ \hline
		\end{tabular}
		\caption{Number of iterations of outer FGMRES solver with exact and inexact block preconditioners for the case study in Example \ref{subsec:example1}. Varying mesh size $h$ while aperture is set to $ \gamma = 1/100 $ and all the permeabilities are set to $ K = K_f = K_{\bm{\nu}} = \bm{I} $.}%
		\label{tab:example1_results_h}
	\end{table}
	
		\begin{table}[htbp]
			\centering
			\begin{tabular}{r|r|r|r||r|r|r|}
				\cline{2-7}
				& \multicolumn{3}{c||}{\textit{Inexact}} & \multicolumn{3}{c|}{\textit{Exact}} \\
				\hline \hline
				\multicolumn{1}{|c||}{$\gamma$} & $\mathcal{M}_D$ & $\mathcal{M}_L$ & $\mathcal{M}_U$ & $\mathcal{B}_D$ & $\mathcal{B}_L$ & $\mathcal{B}_U$ \\ \hline
				\multicolumn{1}{|l||}{$1$} & 21 & 16 & 14 & 21 & 11 & 11  \\ \hline
				\multicolumn{1}{|l||}{$1/10$} & 19 & 13 & 12 & 19 & 10 & 10    \\ \hline
				\multicolumn{1}{|l||}{$1/100$} & 19 & 13 & 11 & 19 & 10 & 10   \\ \hline
				\multicolumn{1}{|l||}{$1/1000$} & 19 & 13 & 11 & 19 & 10 & 10  \\ \hline
				\multicolumn{1}{|l||}{$1/10000$} & 19 & 13 & 11 & 19 & 10 & 10    \\ \hline
			\end{tabular}
			\caption{Number of outer iterations of FGMRES solver with exact and inexact block preconditioners for the case study in Example \ref{subsec:example1}. Varying aperture $\gamma$ while mesh size is set to $ h = 1/16 $ and all the permeabilities are set to $ K = K_f = K_{\bm{\nu}} = \bm{I} $.}%
			\label{tab:example1_results_aperture}
		\end{table}
	
	The results of these robustness tests on are summarized in Tables \ref{tab:example1_results_h} -- \ref{tab:example1_results_perm}. We start with setting $ K_f = K_{\bm{\nu}} = \bm{I} $ that, together with rock permeability $ K $, gives a global homogeneous unitary permeability field. We also fix the aperture to $ \gamma = 10^{-2} $. Refining the initial coarse grid by a factor of 2 recursively, Table \ref{tab:example1_results_h} demonstrates the robustness of all block preconditioners with respect to the mesh size $ h $. Additionally, the different implementations of the preconditioners result in similar behavior of the solver. We notice that the block triangular preconditioners $ \mathcal{B}_L $ and $ \mathcal{B}_U $ show a slightly better performance compared the block diagonal $ \mathcal{B}_D $ as expected. The same behavior can be observed for inexact preconditioners $ \mathcal{M}_L $ and $ \mathcal{M}_U $ in comparison to $ \mathcal{M}_D $. This is expected since the block triangular preconditioners better approximate the inverse of the stiffness matrix in  \eqref{sec:2.eq:block_form_h}. It is noteworthy to mention that the action of the block triangular preconditioners is more expensive computationally than the action of the block diagonal preconditioners.
	Similar performance can also be observed in Table \ref{tab:example1_results_aperture}, where we scale down the fracture width on a fixed grid of mesh size $ h = 1/16 $. Lastly, in Table \ref{tab:example1_results_perm} we test the influence of the heterogeneity in the permeability fields. We keep the mesh size to be $ h = 1/16 $ and fracture aperture to be $ \gamma = 10^{-2} $, while introducing both conducting and blocking fracture network in the porous medium. Again, the robustness is evident in terms of the number of outer FGMRES iterations with both exact and inexact block preconditioners.  The block triangular preconditioners, $ \mathcal{B}_L $, $ \mathcal{B}_U $, $ \mathcal{M}_L $, and $ \mathcal{M}_U $, provide somewhat lower values comparing to their block diagonal counterpart.
		
	\begin{table}[htbp]
		\centering
		\begin{tabular}{r|r|r|r||r|r|r|}
			\cline{2-7}
			& \multicolumn{3}{c||}{\textit{Inexact}} & \multicolumn{3}{c|}{\textit{Exact}} \\
			\hline \hline
			\multicolumn{1}{|c||}{$K$} & $\mathcal{M}_D$ & $\mathcal{M}_L$ & $\mathcal{M}_U$ & $\mathcal{B}_D$ & $\mathcal{B}_L$ & $\mathcal{B}_U$ \\ \hline
			\multicolumn{1}{|l||}{$K_f = 10^{-4}\bm{I}, K_{\bm{\nu}} = 10^{-4}\bm{I}$} & 13 & 10 & 8 & 11 & 7 & 6  \\ \hline
			\multicolumn{1}{|l||}{$K_f = 10^{-4}\bm{I}, K_{\bm{\nu}} = \bm{I}$} & 13 & 8 & 8 & 13 & 7 & 7  \\ \hline
			\multicolumn{1}{|l||}{$K_f = 10^{-4}\bm{I}, K_{\bm{\nu}} = 10^{4}\bm{I}$} & 13 & 8 & 8 & 13 & 7 & 7  \\ \hline
			\multicolumn{1}{|l||}{$K_f = \bm{I}, K_{\bm{\nu}} = 10^{-4}\bm{I}$} & 22 & 16 & 13 & 19 & 11 & 10  \\ \hline
			\multicolumn{1}{|l||}{$K_f = \bm{I}, K_{\bm{\nu}} = \bm{I}$} & 19 & 13 & 11 & 19 & 10 & 10  \\ \hline
			\multicolumn{1}{|l||}{$K_f = \bm{I}, K_{\bm{\nu}} = 10^{4}\bm{I}$} & 19 & 13 & 12 & 19 & 10 & 10  \\ \hline
			\multicolumn{1}{|l||}{$K_f = 10^{4}\bm{I}, K_{\bm{\nu}} = 10^{-4}\bm{I}$} & 26 & 19 & 19 & 21 & 13 & 12  \\ \hline
			\multicolumn{1}{|l||}{$K_f = 10^{4}\bm{I}, K_{\bm{\nu}} = \bm{I}$} & 23 & 17 & 15 & 23 & 13 & 12  \\ \hline
			\multicolumn{1}{|l||}{$K_f = 10^{4}\bm{I}, K_{\bm{\nu}} = 10^{4}\bm{I}$} & 23 & 17 & 15 & 23 & 14 & 12  \\ \hline
		\end{tabular}
		\caption{Number of outer iterations of FGMRES solver with exact and inexact block preconditioners for the case study in Example \ref{subsec:example1}. Varying the permeability $K_f$ and $K_{\bm{\nu}}$ while mesh size is set to $ h = 1/16 $ and aperture is set to $ \gamma = 1/100 $.}%
		\label{tab:example1_results_perm}
	\end{table}


	\subsection{Example: two-dimensional complex network}
	\label{subsec:example2}

	This example is chosen to demonstrate the robustness of the block preconditioners on a more realistic fracture network. Such a complex fracture configuration often occurs in geological rock simulations and the geometrical and physical properties of the fracture network can significantly influence the stability of the solving method. This is especially seen in mpartitioning the fractured porous medium domain where sharp tips and very acute intersections may decrease the shape regularity of the mesh. Since our analysis shows that the performance of our block preconditioners only depends on the shape regularity of the mesh, for this complex network example, we expect to see that the preconditioners are still robust with respect to physical and discretization parameters, but slightly more iterations may be required due to the worse shape regularity of the mesh when comparing to Example~\ref{subsec:example1}.
	
	
	\begin{figure}[htbp]
		\centering
		\includegraphics[valign=c, width=0.435\textwidth]{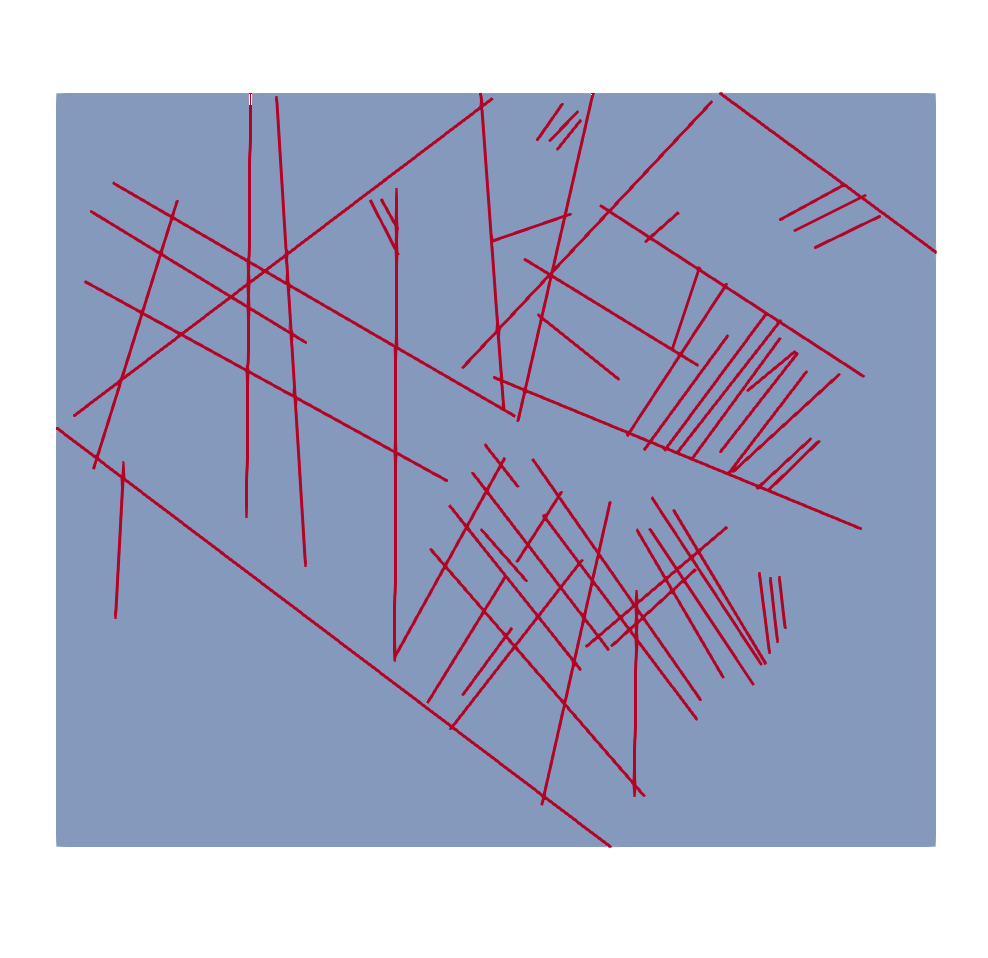}
		\includegraphics[valign=c, width=0.50\textwidth]{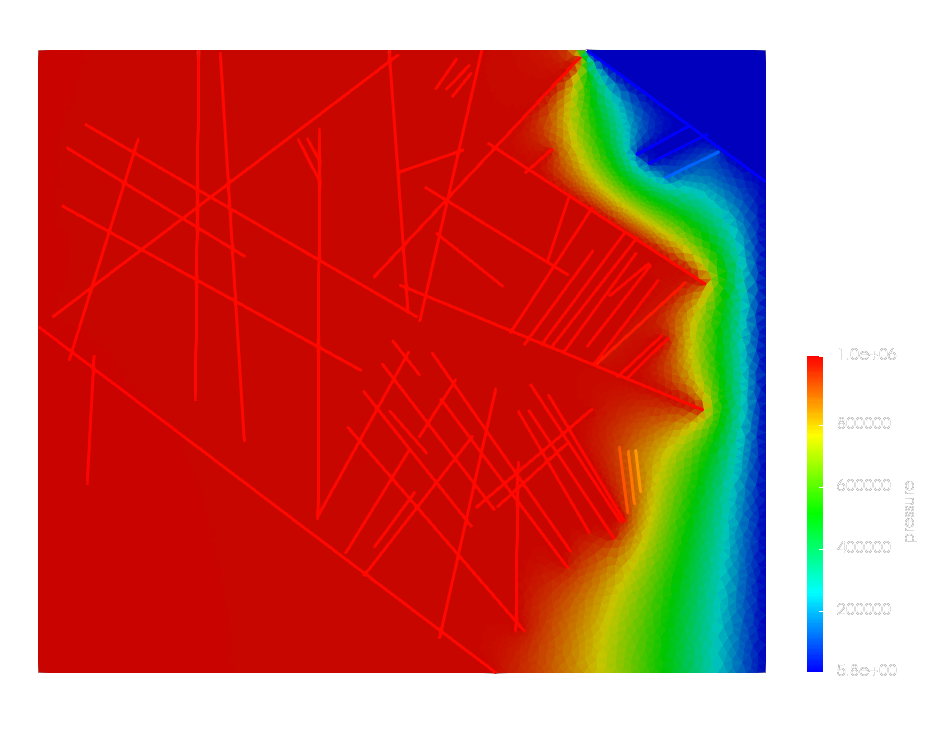}
		\caption{(Left) Graphical representation of the two-dimensional domain and fracture network
			geometry of Example \ref{subsec:example2}. (Right) Pressure solution for a case of conducting fractures.}
		\label{fig:sotra_2d_domain}
	\end{figure} 
	
	This example is chosen from benchmark study \cite{Flemisch2016a} -- a set of fractures from an interpreted outcrop in the Sotra island, near Bergen in Norway. The set includes 64 fractures grouped in 13 different connected networks. The porous medium domain has size $ 700 $ m $ \times $ $ 600 $ m with uniform matrix permeability $ K = 10^{-14} \bm{I} $ m$ ^2 $.  All  the  fractures  have  the  same  scalar  permeability $ K_f = 10^{-8} \bm{I} $ m$ ^2 $ and aperture $ \gamma = 10^{-2} $ m.  Also, no-flow boundary condition are imposed on top and bottom, with pressure $ 1013250 $ Pa on the left and $ 0 $ Pa on the right boundary. 
	
	\begin{figure}[htbp]
		\centering
		\includegraphics[valign=c, width=0.35\textwidth]{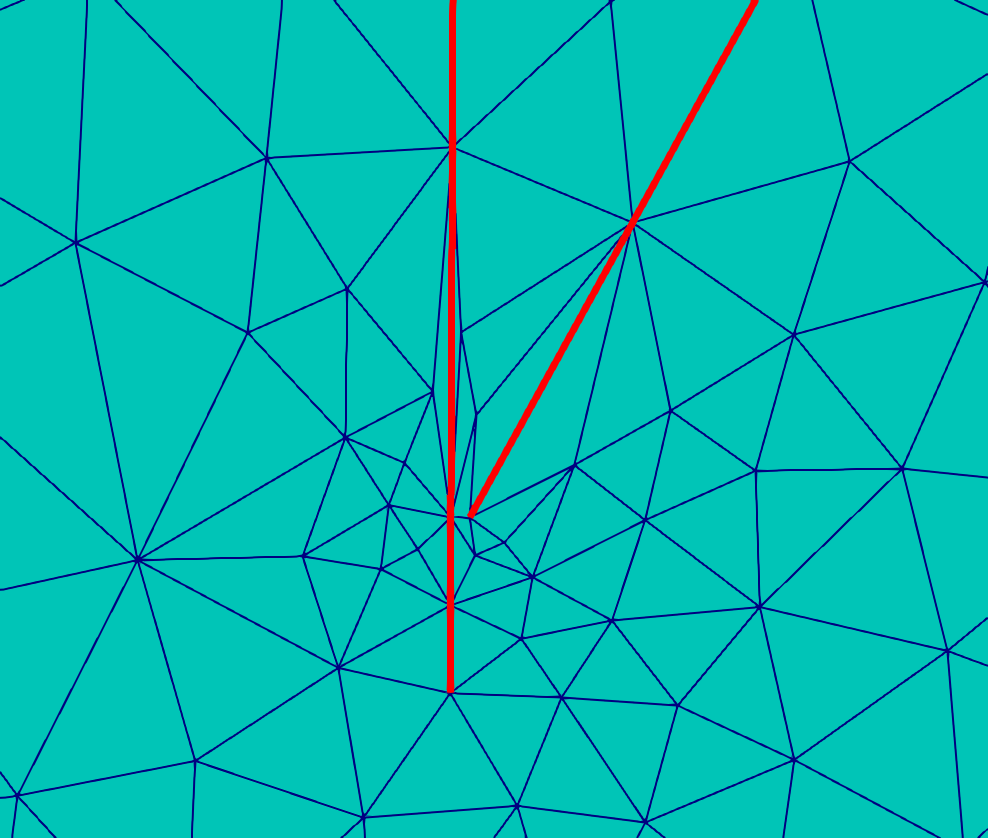}
		\hskip 1cm
		\begin{tabular}{r|r|r|r|}
			\cline{2-4}
			& \multicolumn{3}{c|}{\textit{Inexact}} \\
			\hline \hline
			\multicolumn{1}{|c||}{$h$} & $\mathcal{M}_D$ & $\mathcal{M}_L$ & $\mathcal{M}_U$ \\ \hline
			\multicolumn{1}{|l||}{$L/4$} & 63 & 51 & 40   \\ \hline
			\multicolumn{1}{|l||}{$L/8$} & 67 & 50 & 44   \\ \hline
			\multicolumn{1}{|l||}{$L/16$} & 61 & 47 & 42    \\ \hline
			\multicolumn{1}{|l||}{$L/32$} & 55 & 39 & 34     \\ \hline
			\multicolumn{1}{|l||}{$L/64$} & 47 & 33 & 29     \\ \hline
		\end{tabular}
		\caption{(Left) Mesh around one of the complex tips in the fracture network, where $ h = L/64 $. (Right) Number of outer iterations of FGMRES solver with exact and inexact block preconditioners for the case study in Example \ref{subsec:example2}. We refine the mesh relatively to domain length $ L = 600 $.}%
		\label{fig:example2_results_h}
	\end{figure}

	For the comparison with the previous example, we refine the mesh size $ h $ with respect to the width of the domain $ L = 600 $. However, due to the complex fracture structure, it is possible to end up with smaller and badly shaped elements in the rock matrix grid around the tips and intersections of the fractures. For example, see Figure~\ref{fig:example2_results_h} on the left. The coarser the mesh is, the more irregular the elements are, especially when partisioning in between many tightly packed fractures. Therefore, we expect that the solver requires more iterations to converge on coarser meshes. This is evident in the table on the right in Figure~\ref{fig:example2_results_h}. We see the reduction of number of iterations when refining the mesh in all the cases, with the lowest number required by the block upper triangular $ \mathcal{M}_U $.  We also notice that the solver manages to provide the correct solution on all given meshes in an acceptable number of iterations. The results are slightly worse than the previous example, but keep in mind that the complex geometry is still an important factor in the mesh structure and, therefore, influences the convergence rate since the shape regularity of the mesh deteriorates. For complex fracture networks, it is beneficial to invest in constructing a more regular mesh of the fractured porous medium and then applying the proposed block preconditioners in the iterative solvers. 

	
	\subsection{Example: three-dimensional Geiger network}
	\label{subsec:example3}
	
	\begin{figure}[htbp]
		\centering
		\includegraphics[width=0.45\textwidth]{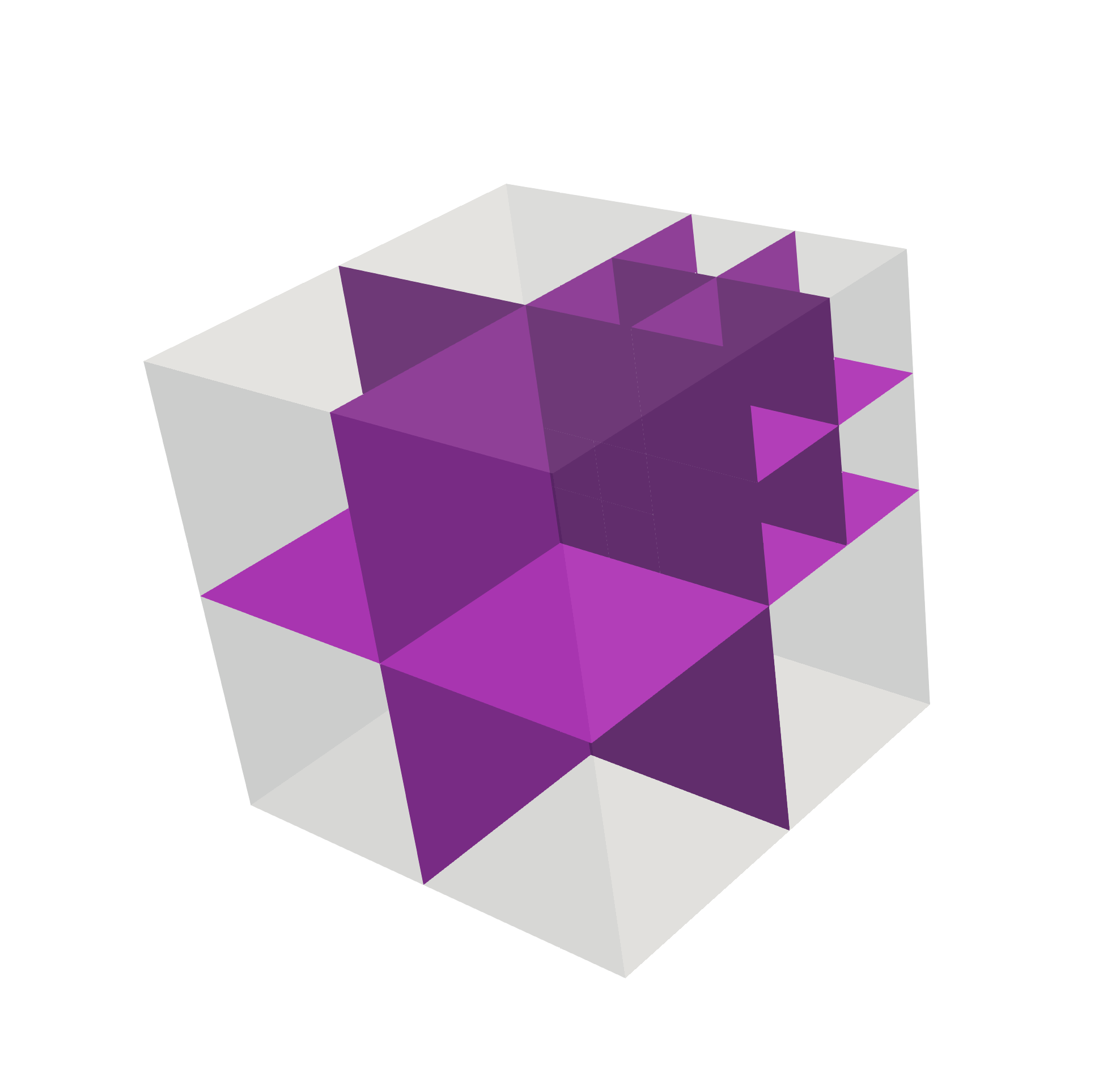}
		\includegraphics[width=0.45\textwidth]{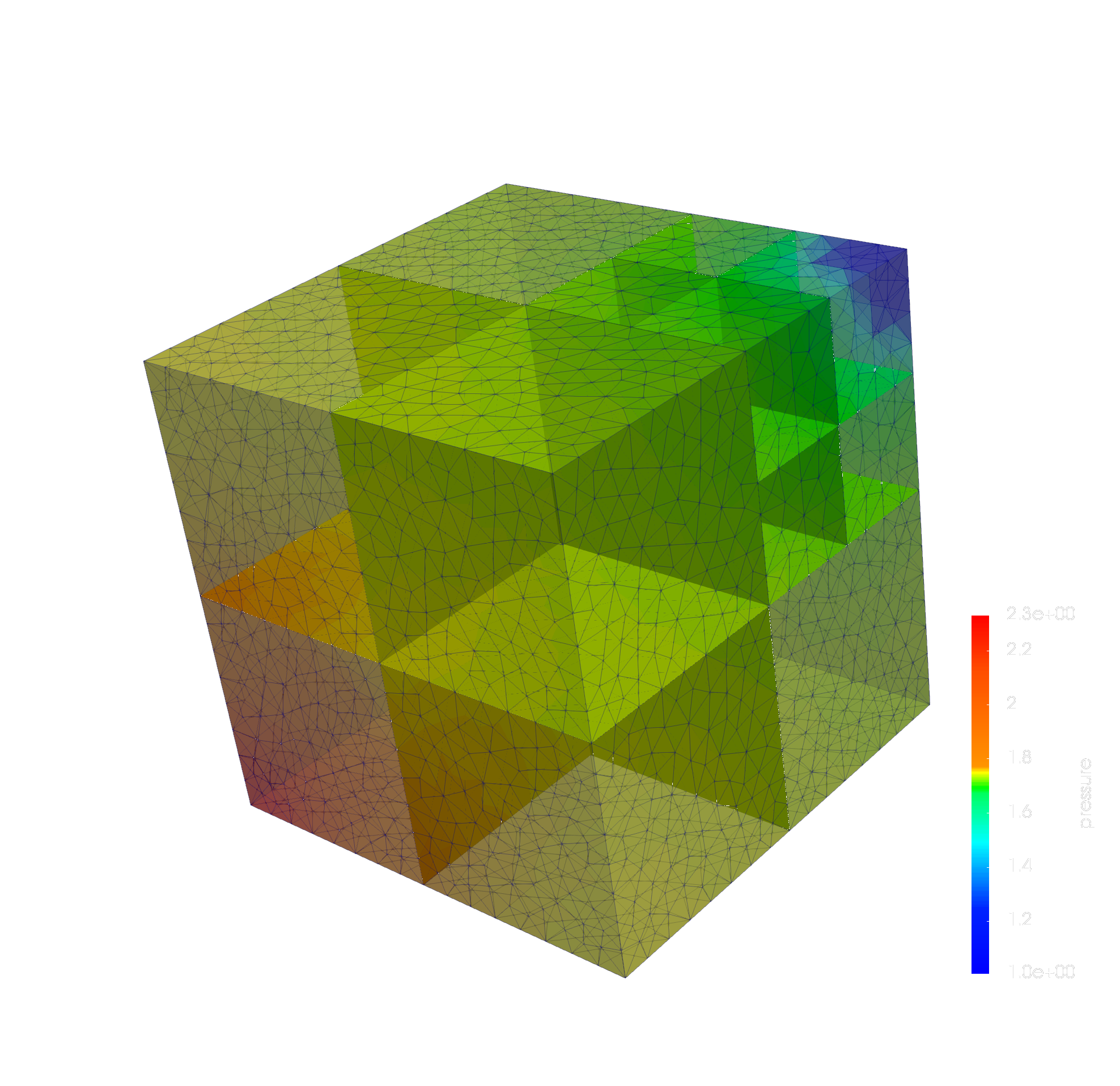}
		\caption{(Left) Graphical representation of the three-dimensional domain and fracture network
			geometry of Example \ref{subsec:example3}. (Right) Pressure solution for a case of conducting fractures.}
		\label{fig:geiger_3d_domain}
	\end{figure}
	
	This last example considers the simulations of a 3D problem taken from another benchmark study \cite{berre2018}, a three-dimensional analogue to the test case in Subsection \ref{subsec:example1}. The geometry is extended to the unit cube and the fracture network now consists of nine intersecting planes (see Figure \ref{fig:geiger_3d_domain}). As before, we take the rock matrix permeability $ K $ to be the identity tensor, while we vary the tangential $ K_f $ and the normal $ K_{\bm{\nu}} $ permeability, as well as the fracture aperture $ \gamma $. 
	
	For a fair comparison with the two-dimensional case, we perform similar robustness tests of the preconditioners to study the effect of mesh refinement, as well as permeability and aperture changes. However, we stick to only inexact preconditioners $ \mathcal{M}_D $, $ \mathcal{M}_L $ and $ \mathcal{M}_U $ since they are less computationally expensive and perform comparably well, which makes them good choices in practice. The results are presented in Tables \ref{tab:example3_results_h}--\ref{tab:example3_results_perm}. We can see that the simulations confirm the findings of Section \ref{sec:precond_MD}: all block preconditioners show robustness with respect to the discretization and physical parameters. The block diagonal preconditioner requires a slightly higher number of iterations to converge compared to block triangular ones, as we saw in the previous example.
	
	\begin{table}[htbp]
		\centering
		\begin{tabular}{r|r|r|r|}
			\cline{2-4}
			& \multicolumn{3}{c|}{\textit{Inexact}} \\
			\hline \hline
			\multicolumn{1}{|c||}{$h$} & $\mathcal{M}_D$ & $\mathcal{M}_L$ & $\mathcal{M}_U$ \\ \hline
			\multicolumn{1}{|l||}{$1/4$} & 26 & 18 & 15  \\ \hline
			\multicolumn{1}{|l||}{$1/8$} & 26 & 17 & 15    \\ \hline
			\multicolumn{1}{|l||}{$1/16$} & 24 & 16 & 14    \\ \hline
			\multicolumn{1}{|l||}{$1/32$} & 24 & 16 & 13     \\ \hline
			\multicolumn{1}{|l||}{$1/64$} & 24 & 16 & 12     \\ \hline
		\end{tabular}
		\caption{Number of outer iterations of FGMRES solver with exact and inexact block preconditioners for the case study in Example \ref{subsec:example2}. Varying mesh size $h$ while aperture is set to $ \gamma = 1/100 $ and all permeabilities are set to $ K = K_f = K_{\bm{\nu}} = \bm{I} $.}%
		\label{tab:example3_results_h}
	\end{table}

	\begin{table}[htbp]
		\centering
		\begin{tabular}{r|r|r|r|}
			\cline{2-4}
			& \multicolumn{3}{c|}{\textit{Inexact}} \\
			\hline \hline
			\multicolumn{1}{|c||}{$\gamma$} & $\mathcal{M}_D$ & $\mathcal{M}_L$ & $\mathcal{M}_U$ \\ \hline
			\multicolumn{1}{|l||}{$1$} & 24 & 16 & 14  \\ \hline
			\multicolumn{1}{|l||}{$1/10$} & 24 & 16 & 13  \\ \hline
			\multicolumn{1}{|l||}{$1/100$} & 24 & 16 & 14  \\ \hline
			\multicolumn{1}{|l||}{$1/1000$} & 26 & 16 & 14  \\ \hline
			\multicolumn{1}{|l||}{$1/10000$} & 26 & 17 & 14  \\ \hline
		\end{tabular}
		\caption{Number of outer iterations of FGMRES solver with exact and inexact block preconditioners for the case study in Example \ref{subsec:example3}. Varying aperture $\gamma$ while mesh size is set to $ h = 1/16 $ and all permeabilities are set to $ K = K_f = K_{\bm{\nu}} = \bm{I} $.}%
		\label{tab:example3_results_aperture}
	\end{table}

	\begin{table}[htbp]
		\centering
		\begin{tabular}{r|r|r|r|}
			\cline{2-4}
			& \multicolumn{3}{c|}{\textit{Inexact}} \\
			\hline \hline
			\multicolumn{1}{|c||}{$K$} & $\mathcal{M}_D$ & $\mathcal{M}_L$ & $\mathcal{M}_U$ \\ \hline
			\multicolumn{1}{|l||}{$K_f = 10^{-4}\bm{I}, K_{\bm{\nu}} = 10^{-4}\bm{I}$} & 28 & 19 & 20 \\ \hline
			\multicolumn{1}{|l||}{$K_f = 10^{-4}\bm{I}, K_{\bm{\nu}} = \bm{I}$} & 26 & 17 & 14  \\ \hline
			\multicolumn{1}{|l||}{$K_f = 10^{-4}\bm{I}, K_{\bm{\nu}} = 10^{4}\bm{I}$} & 28 & 17 & 14 \\ \hline
			\multicolumn{1}{|l||}{$K_f = \bm{I}, K_{\bm{\nu}} = 10^{-4}\bm{I}$} & 26 & 21 & 18  \\ \hline
			\multicolumn{1}{|l||}{$K_f = \bm{I}, K_{\bm{\nu}} = \bm{I}$} & 24 & 16 & 14 \\ \hline
			\multicolumn{1}{|l||}{$K_f = \bm{I}, K_{\bm{\nu}} = 10^{4}\bm{I}$} & 26 & 17 & 14 \\ \hline
			\multicolumn{1}{|l||}{$K_f = 10^{4}\bm{I}, K_{\bm{\nu}} = 10^{-4}\bm{I}$} & 24 & 16 & 17 \\ \hline
			\multicolumn{1}{|l||}{$K_f = 10^{4}\bm{I}, K_{\bm{\nu}} = \bm{I}$} & 22 & 15 & 13 \\ \hline
			\multicolumn{1}{|l||}{$K_f = 10^{4}\bm{I}, K_{\bm{\nu}} = 10^{4}\bm{I}$} & 22 & 15 & 13 \\ \hline
		\end{tabular}
		\caption{Number of outer iterations of FGMRES solver with exact and inexact block preconditioners for the case study in Example \ref{subsec:example3}. Varying the permeability $K_f$ and $K_{\bm{\nu}}$ while mesh size is set to $ h = 1/16 $ and aperture is set to $ \gamma = 1/100 $.}%
		\label{tab:example3_results_perm}
	\end{table}
	
	In 3D simulations it is also important to study the overall computational complexity of the solving method. For that, we analyze in Figure \ref{fig:cpu_time} the required CPU time of the FGMRES solver preconditioned with each block preconditioner $ \mathcal{M}_D $, $ \mathcal{M}_L $ and $ \mathcal{M}_U $. All preconditioners show a optimal $ \mathcal{O}(N_{dof}) $ complexity, where $ N_{dof} $ is the number of degrees of freedom of the discretized system. Notice that even though the block triangular pair of preconditioners require solving a denser system, it is still time-wise less expensive due to a lower number of iterations needed to converge.
	
	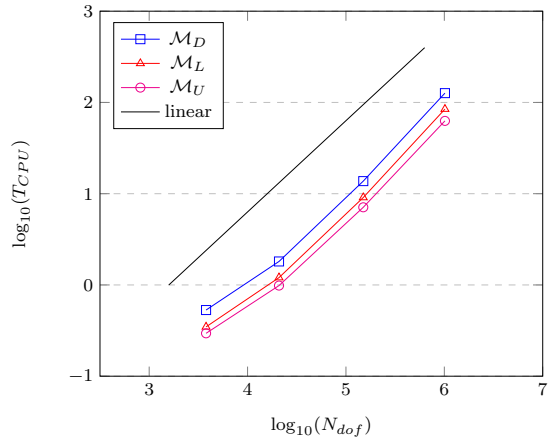
\begin{figure}[htbp]
		\centering 
		\scalebox{0.85}{
			\begin{tikzpicture}
			\begin{axis}[
			xlabel={$\log_{10}(N_{dof})$},
			ylabel={$\log_{10}(T_{CPU})$},
			xmin=2.5, xmax=7,
			ymin=-1, ymax=3,
			xtick={3, 4, 5, 6, 7},
			ytick={-1, 0, 1, 2, 3},
			legend pos=north west,
			ymajorgrids=true,
			grid style=dashed,
			]
			
			\addplot[
			color=blue,
			mark=square,
			]
			coordinates {
				(3.5781, -0.2751)(4.3196, 0.2574)(5.1761, 1.1378)(6.0062, 2.1020)
			};
			\addlegendentry{$\mathcal{M}_D$}
			
			\addplot[
			color=red,
			mark=triangle,
			]
			coordinates {
				(3.5781, -0.4595)(4.3196, 0.0805)(5.1761, 0.9582)(6.0062, 1.9269)
			};
			\addlegendentry{$\mathcal{M}_L$}
			
			\addplot[
			color=magenta,
			mark=o,
			]
			coordinates {
				(3.5781, -0.5302)(4.3196, -0.0069)(5.1761, 0.8509)(6.0062, 1.7978)
			};
			\addlegendentry{$\mathcal{M}_U$}
			
			\addplot[
			domain=3.2:5.8,
			samples=40,
			color=black,
			]
			{x - 3.2};
			\addlegendentry{linear}
			
			\end{axis}
			\end{tikzpicture}
		}
		\caption{CPU time $ T_{CPU} $ of the preconditioned FGMRES algorithm verses number of degrees of freedom $ N_{dof} $ of the discretized system in Example \ref{subsec:example3}.}
		\label{fig:cpu_time}
	\end{figure}


	\section{Conclusions}
	\label{sec:conclusion}
	
	We have presented block preconditioners for linear systems arising in mixed-dimensional modeling of single-phase flow in fractured porous media. Our approach is based on the stability theory of the mixed finite element discretization of the model which we extended to provide an efficient way to solve large systems with standard Krylov subspace iterative methods. We have thoroughly analyzed the robustness of the derived preconditioners with regard to discretization and physical parameters by proving norm and field-of-value equivalence to the original system. Our theory has also been supported by several numerical examples of 2D and 3D flow simulations. 
	
	It is noteworthy to mention that even though our analysis depends on a more regular mesh, the numerical results show that the preconditioners still perform well since the mixed-dimensional discretization approach handles fractures independently of the rock matrix and, therefore, generates simpler meshes in most fracture network cases. The large aspect ratios that parametrize the model then become the main stability problem, which we have successfully overcome with the proposed block preconditioners. This is important for implementations in general geological simulations where the rock-fracture configuration can be quite complex and can contain a large number of fractures of different width and length.
	
	We conclude by recalling that the alternative approach to block preconditioners mentioned in the beginning of Section \ref{sec:precond} is a non-trivial extension to this work and a part of an ongoing research.


	\section{Acknowledgements}
	A special thanks is extended to James Adler, Alessio Fumagalli and Eirik Keilegavlen for valuable comments and discussions on the presented work.  The authors also would like to thank Casey Cavanaugh for improving the style of the presentation.

		
	\bibliographystyle{spmpsci}
	\bibliography{references}

\end{document}